\documentclass[12pt]{amsart}
\usepackage{amsmath,amssymb,amsbsy,amsfonts,latexsym,amsopn,amstext,
                                               amsxtra,euscript,amscd,bm}
                    
\usepackage{url}
\usepackage[colorlinks,linkcolor=blue,anchorcolor=blue,citecolor=blue]{hyperref}
\usepackage{color}
\begin{document}
\newtheorem{problem}{Problem}
\newtheorem{theorem}{Theorem}
\newtheorem{lemma}[theorem]{Lemma}
\newtheorem{claim}[theorem]{Claim}
\newtheorem{cor}[theorem]{Corollary}
\newtheorem{prop}[theorem]{Proposition}
\newtheorem{definition}{Definition}
\newtheorem{question}[theorem]{Question}

%%%%%%%%%%%%%%%%%%%%%%%%%
% Alphabet calligraphic %
%%%%%%%%%%%%%%%%%%%%%%%%%
\def\cA{{\mathcal A}}
\def\cB{{\mathcal B}}
\def\cC{{\mathcal C}}
\def\cD{{\mathcal D}}
\def\cE{{\mathcal E}}
\def\cF{{\mathcal F}}
\def\cG{{\mathcal G}}
\def\cH{{\mathcal H}}
\def\cI{{\mathcal I}}
\def\cJ{{\mathcal J}}
\def\cK{{\mathcal K}}
\def\cL{{\mathcal L}}
\def\cM{{\mathcal M}}
\def\cN{{\mathcal N}}
\def\cO{{\mathcal O}}
\def\cP{{\mathcal P}}
\def\cQ{{\mathcal Q}}
\def\cR{{\mathcal R}}
\def\cS{{\mathcal S}}
\def\cT{{\mathcal T}}
\def\cU{{\mathcal U}}
\def\cV{{\mathcal V}}
\def\cW{{\mathcal W}}
\def\cX{{\mathcal X}}
\def\cY{{\mathcal Y}}
\def\cZ{{\mathcal Z}}

%%%%%%%%%%%%%%%%%%%%%%%
% Alphabet blackboard %
%%%%%%%%%%%%%%%%%%%%%%%
\def\A{{\mathbb A}}
\def\B{{\mathbb B}}
\def\C{{\mathbb C}}
\def\D{{\mathbb D}}
\def\E{{\mathbb E}}
\def\F{{\mathbb F}}
\def\G{{\mathbb G}}
\def\I{{\mathbb I}}
\def\J{{\mathbb J}}
\def\K{{\mathbb K}}
\def\L{{\mathbb L}}
\def\M{{\mathbb M}}
\def\N{{\mathbb N}}
\def\O{{\mathbb O}}
\def\P{{\mathbb P}}
\def\Q{{\mathbb Q}}
\def\R{{\mathbb R}}
\def\S{{\mathbb S}}
\def\T{{\mathbb T}}
\def\U{{\mathbb U}}
\def\V{{\mathbb V}}
\def\W{{\mathbb W}}
\def\X{{\mathbb X}}
\def\Y{{\mathbb Y}}
\def\Z{{\mathbb Z}}

\def\ep{{\mathbf{e}}_p}
\def\em{{\mathbf{e}}_m}
\def\eq{{\mathbf{e}}_q}

\def\scr{\scriptstyle}
\def\\{\cr}
\def\({\left(}
\def\){\right)}
\def\[{\left[}
\def\]{\right]}
\def\<{\langle}
\def\>{\rangle}
\def\fl#1{\left\lfloor#1\right\rfloor}
\def\rf#1{\left\lceil#1\right\rceil}
\def\le{\leqslant}
\def\ge{\geqslant}
\def\eps{\varepsilon}
\def\mand{\qquad\mbox{and}\qquad}

\def\sssum{\mathop{\sum\ \sum\ \sum}}
\def\ssum{\mathop{\sum\, \sum}}
\def\ssumw{\mathop{\sum\qquad \sum}}

\def\vec#1{\mathbf{#1}}
\def\inv#1{\overline{#1}}
\def\num#1{\mathrm{num}(#1)}
\def\dist{\mathrm{dist}}

\def\fA{{\mathfrak A}}
\def\fB{{\mathfrak B}}
\def\fC{{\mathfrak C}}
\def\fU{{\mathfrak U}}
\def\fV{{\mathfrak V}}

\newcommand{\bflambda}{{\boldsymbol{\lambda}}}
\newcommand{\bfxi}{{\boldsymbol{\xi}}}
\newcommand{\bfrho}{{\boldsymbol{\rho}}}
\newcommand{\bfnu}{{\boldsymbol{\nu}}}

\def\GL{\mathrm{GL}}
\def\SL{\mathrm{SL}}

\def\Hba{\overline{\cH}_{a,m}}
\def\Hta{\widetilde{\cH}_{a,m}}
\def\Hb1{\overline{\cH}_{m}}
\def\Ht1{\widetilde{\cH}_{m}}

\def\flp#1{{\left\langle#1\right\rangle}_p}
\def\flm#1{{\left\langle#1\right\rangle}_m}
\def\dmod#1#2{\left\|#1\right\|_{#2}}
\def\dmodq#1{\left\|#1\right\|_q}

\def\Zm{\Z/m\Z}

\def\Err{{\mathbf{E}}}

\newcommand{\comm}[1]{\marginpar{%
\vskip-\baselineskip %raise the marginpar a bit
\raggedright\footnotesize
\itshape\hrule\smallskip#1\par\smallskip\hrule}}

\def\xxx{\vskip5pt\hrule\vskip5pt}

%%%%%%%%%%%%%%%%%%
%% PAPER BEGINS %%
%%%%%%%%%%%%%%%%%%

\title{On the cubic Weyl sum}

 \author[B. Kerr] {Bryce Kerr}
\address{Max Planck Institute for Mathematics, Bonn, Germany}
\email{bryce.kerr89@gmail.com}

%\keywords{??}
%\subjclass[2010]{??}
\begin{abstract}  
We obtain an estimate for the cubic Weyl sum which improves the bound obtained from Weyl differencing for short ranges of summation. In particular, we show that for any $\varepsilon>0$ there exists some $\delta>0$ such that for any coprime integers $a,q$ and real number $\gamma$  we have 
\begin{align*}
\sum_{1\le n \le N}e\left(\frac{an^3}{q}+\gamma n\right)\ll (qN)^{1/4}q^{-\delta},
\end{align*}
provided $q^{1/3+\varepsilon}\le N \le q^{1/2-\varepsilon}$. Our argument builds on some ideas of Enflo.
\end{abstract} 

\maketitle
\section{Introduction}
Let $g(x)=\alpha_d x^d+\dots+\alpha_1 x$ be a polynomial of degree $d\ge 2$ with real coefficients and consider the problem of bounding sums
\begin{align*}
\sum_{1\le n \le N}e(g(n)).
\end{align*}
This was first considered by Weyl~\cite{Weyl} who developed a technique known as Weyl differencing which shows 
\begin{equation}
\label{eq:weyl}
\sum_{1\le n \le N}e(g(n))\ll N^{1+o(1)}\left(\frac{1}{q}+\frac{1}{N}+\frac{q}{N^d}\right)^{2^{-(d-1)}},
\end{equation}
provided the leading coefficient $\alpha_d$ of $g$ satisfies 
\begin{align*}
\left|\alpha_d-\frac{a}{q}\right|\le \frac{1}{q^2},
\end{align*}
for some coprime integers $a,q$, see~\cite[Lemma~2.4]{Vau}. For large values of $d$ the estimate~\eqref{eq:weyl} has been improved via bounds related to the Vinogradov mean value theorem while~\eqref{eq:weyl} remains sharpest for $d\le 5$. 

 In this paper we consider the problem of improving the bound~\eqref{eq:weyl} in the case of monomials with $d=3$. In particular, our sums take the form
\begin{align}
\label{eq:cb1}
\sum_{1\le n \le N}e\left(\frac{an^3}{q}\right).
\end{align}
This problem has been considered by a number of previous authors. Enflo~\cite{Enflo} has some nice discussions on different approaches and parts of our ideas and analysis are based on this work. We mention an open problem from~\cite[Section~2]{Enflo} to improve on the P\'{o}lya-Vinogradov bound for sums
\begin{align}
\label{eq:pv123}
\sum_{1\le n \le N}\chi(n)e\left(\frac{an^3+b\overline{n}+cn}{q} \right)\ll q^{1/2+o(1)},
\end{align}
where $\chi$ is a multiplicative character mod $q$. Any improvement on~\eqref{eq:pv123} would give a new bound for~\eqref{eq:cb1} in certain ranges of parameters.

 Heath-Brown~\cite{HB} has applied the $q$-van der Corput method to the sums~\eqref{eq:cb1} 
and showed how one may improve~\eqref{eq:weyl} provided $q$ has suitable factorization. 

We refer the reader to Nakai~\cite{Nakai} for some ideas related to generalizing the classic Hardy-Littlewood argument~\cite{HL} from the setting of quadratic to cubic sums.

 In this paper we obtain a new bound for cubic sums with a short range of summation.
 \begin{theorem}
\label{thm:main1}
Let $\delta>0$, $a,q$ coprime integers and $\gamma$ a real number. If $N$ satisfies 
\begin{align*}
q^{1/3+\delta}\ll N \ll q^{1/2-\delta/10+o(1)},
\end{align*}
then 
\begin{align*}
\sum_{1\le n \le N}e\left(\frac{an^3}{q}+\gamma n \right)\ll (qN)^{1/4+o(1)}q^{-\delta/20}.
\end{align*}
\end{theorem}
Note that~\eqref{eq:weyl} implies 
\begin{align*}
\sum_{1\le n \le N}e\left(\frac{an^3}{q}+\gamma n \right)\ll (qN)^{1/4+o(1)} \quad \text{provided} \quad q^{1/3}\le N \le q^{1/2},
\end{align*}
hence we obtain a power saving over Weyl differencing in this range of parameters.

We are concerned with bounding the sums~\eqref{eq:cb1} for individual values of $q$. There are also a number of interesting results and questions concerning their average or typical behaviour. Wooley's~\cite{Wooley} work on the cubic case of Vinogradov's mean value theorem implies sharp bounds for moments of sums over cubic polynomials, see also~\cite{Wooley1}. More general restriction estimates for cubic sums were initiated by Bourgain~\cite{Bou} and further developed in~\cite{Hu,HuW,Lai}. See~\cite{CS,CS1} for various bounds related to the metric behaviour of large values of Weyl sums and~\cite{BPPSV} for other metric results and an asymptotic formula for minor arc sums.

\section{Outline of argument}
The details in proving Theorem~\ref{thm:main1} are sometimes tedious so we first outline our argument using some heuristics to simplify things. For the purpose of this discussion we assume a heuristic form of Poission summation, that for any $q$-periodic function $f$ 
\begin{align}
\label{eq:heuristic1}
\sum_{n \le N}f(n)\sim \frac{N}{q}\sum_{n\le q/N}\widehat f(n).
\end{align}
where $\widehat f$ is the Fourier transform 
\begin{align*}
\widehat f(n)=\sum_{y=1}^{q}f(y)e_q(-yn),
\end{align*} 
and $e_q(x)=e^{2\pi i x/q}.$
Let $G(a,b;q)$ denote the Gauss sum
\begin{align*}
G(a,b;q)=\sum_{y=1}^{q}e_q(ay^2+by).
\end{align*}
Assuming $q$ is odd, we may complete the square to obtain 
\begin{align*}
G(a,b;q)=e_q\left(-\overline{4a}b^2\right)G(a,0;q).
\end{align*}
The sums $G(a,0;q)$ have an explicit evaluation in terms of the Jacobi symbol. To state such results requires considering a case by case basis depending on the value of $q$ mod $4$. We will ignore issues with even $q$ and for simplicity  assume for any integers $a,b,q$ that
\begin{align}
\label{eq:heuristic2}
G(a,b;q)\sim \left(\frac{a}{q}\right)e_q\left(-\overline{4a}b^2\right)q^{1/2},
\end{align}
where 
$$\left(\frac{.}{q}\right),$$
denotes the Jacobi symbol mod $q$.

 Let $N$ satisfy 
\begin{align}
\label{eq:heuristicN}
q^{1/3+\varepsilon}\le N \le q^{1/2-\varepsilon},
\end{align}
and consider the sums
\begin{align}
\label{eq:heuristicS}
S=\sum_{1\le n \le N}e_q\left(an^3\right).
\end{align}
In this setting one may use Weyl differencing to show 
\begin{align}
\label{eq:S4141414141}
|S|^4\ll N^{2+o(1)}\left|\left\{ |m|\ll N^2, \quad |n| \ll \frac{q}{N} \ : \ am\equiv n \mod{q}  \right\}\right|.
\end{align}
If
\begin{align}
\label{eq:Subheuristic}
S\gg \delta (qN)^{1/4+o(1)},
\end{align}
for some  $\delta \sim q^{-\varepsilon_0}$
 then
\begin{align}
\label{eq:latticeSMH}
\left|\left\{ |m|\ll N^2, \quad |n| \ll \frac{q}{N} \ : \ am\equiv n \mod{q}  \right\}\right|\gg \delta^4\frac{q}{N}.
\end{align}
Using lattice basis reduction we obtain integers $\ell,s$ satisfying 
\begin{align}
\label{eq:ls1}
a\equiv \ell \overline{s} \mod{q}, \quad \ell  \ll \frac{1}{\delta^4}, \quad |s|\ll \frac{N^3}{q \delta^4}.
\end{align}
Returning to the sums~\eqref{eq:heuristicS}, we apply the first step of Weyl differencing to get 
\begin{align*}
|S|^2=\sum_{m,n\le N}e_q(a(m^3-n^3))&=\sum_{|m|\le N}e_q(am^3)\sum_{n\le N}e_q(3am n^2+3am^2n) \\
&\ll S_1+S_2,
\end{align*}
where 
\begin{align*}
S_1=\sum_{|m|\le N^{1-\varepsilon}}\left|\sum_{n\le N}e_q(3am n^2+3am^2n)\right|,
\end{align*}
and
\begin{align}
\label{eq:S2heur}
S_2=\sum_{N^{1-\varepsilon}<|m|\le N}e_q(am^3)\sum_{n\le N}e_q(3am n^2+3am^2n).
\end{align}
If $|S|^2\ll S_1$ then we continue with Weyl differencing and obtain an immediate improvement on~\eqref{eq:weyl} due to the shorter outer summation. Hence we may suppose 
\begin{align}
\label{eq:SS12h}
|S|^2\ll S_2.
\end{align}
By~\eqref{eq:heuristic1} and~\eqref{eq:heuristic2}
\begin{align}
\label{eq:heuristicS2}
\nonumber S_2&\sim \frac{N}{q}\sum_{N^{1-\varepsilon}<|m|\le N}e_q(am^3)\sum_{n\le q/N}G(3am,3am^2-n;q) \\
&\sim \frac{N}{q^{1/2}}\sum_{N^{1-\varepsilon}<|m|\le N}\left(\frac{m}{q}\right)e_q(\overline{4}am^3)\sum_{n\le q/N}e_q\left(-\overline{12am }n^2+\overline{2}mn\right).
\end{align}
Up to this point our analysis of $S_2$ is similar to ideas used by Enflo~\cite{Enflo} and note the sums~\eqref{eq:pv123} arise after interchanging summation. Our new input is to use~\eqref{eq:ls1} and reciprocity for modular inverses to reduce oscillations for summation over $n$ in~\eqref{eq:heuristicS2}.  We have 
\begin{align*}
\frac{\overline{12am }n^2}{q}\equiv \frac{s\overline{12\ell m }n^2}{q}\equiv -\frac{s\overline{q}n^2}{12\ell m}+\frac{s n^2}{12 \ell m q} \mod{1},
\end{align*}
and 
\begin{align*}
\frac{\overline{2}mn}{q}\equiv -\frac{qmn}{2}+\frac{mn}{2q} \mod{1}.
\end{align*}
Substituting into~\eqref{eq:heuristicS2} gives
\begin{align}
\label{eq:S2b1heuristic}
& \nonumber S_2\sim \frac{N}{q^{1/2}}\sum_{N^{1-\varepsilon}<|m|\le N}\left(\frac{m}{q}\right)e_q(\overline{4}am^3) \\ & \quad \quad \quad \quad \quad \quad \quad \times \sum_{n\le q/N}e_{12\ell m}\left(s\overline{q }n^2\right)e\left(-\frac{sn^2}{12 \ell m q}+\frac{mn}{2q}-\frac{qmn}{2}\right).
\end{align}
If $m,n$ satisfy the conditions of summation in~\eqref{eq:S2b1heuristic} then from~\eqref{eq:ls1}
\begin{align*}
-\frac{sn^2}{12 \ell m q}+\frac{mn}{2q}\ll \frac{N^{\varepsilon}}{\delta^4}.
\end{align*}
Using some Fourier analysis, this implies 
\begin{align}
\label{eq:S2b1heuristic2}
S_2\ll \frac{N^{1+\varepsilon}}{\delta^4q^{1/2}}\sum_{N^{1-\varepsilon}<|m|\le N}\left(\frac{m}{q}\right)e_q(\overline{4}am^3)\sum_{n\le q/N}e_{12\ell m}\left(s\overline{q }n^2\right),
\end{align}
and we have ignored the term $qmn/2$ which may be dealt with by partitioning summation depending on parity. If $N$ satisfies~\eqref{eq:heuristicN} then by~\eqref{eq:ls1} 
\begin{align}
\label{eq:qN3lm}
\frac{q}{N}\gg \ell m, 
\end{align} 
provided $\delta$ is not too small. We see that summation over $n$ in~\eqref{eq:S2b1heuristic2} is essentially a complete sum
\begin{align*}
\sum_{n\le q/N}e_{12\ell m}\left(s\overline{q}n^2\right)\sim  \frac{q}{12\ell m N}G(s\overline{q},0;12\ell m).
\end{align*}
Substituting the above into~\eqref{eq:S2b1heuristic2} and using~\eqref{eq:heuristic2} 
\begin{align*}
S_2\ll \frac{q^{1/2}N^{\varepsilon}}{\delta^4}\sum_{N^{1-\varepsilon}<|m|\le N}\left(\frac{m}{q}\right)\left(\frac{q}{m}\right)\frac{e_q(\overline{4}am^3)}{m^{1/2}}.
\end{align*}
By quadratic reciprocity and partial summation this simplifies to 
\begin{align}
\label{eq:S2finalH}
S_2\ll \frac{q^{1/2}N^{2\varepsilon}}{\delta^4N^{1/2}}\sum_{m\le N}e_q(\overline{4}am^3).
\end{align}
We arrive at a kind of recursive inequality for cubic sums. Chossing $a$ such that the sum~\eqref{eq:heuristicS} is maximum we obtain from~\eqref{eq:SS12h}
\begin{align*}
|S|\ll \frac{q^{1/2}N^{2\varepsilon}}{\delta^4N^{1/2}},
\end{align*} 
and hence by~\eqref{eq:Subheuristic}
\begin{align*}
\delta \ll N^{2\varepsilon/5}\left(\frac{q}{N^{3}}\right)^{1/20},
\end{align*}
which gives a power saving over~\eqref{eq:weyl} provided $N\ge q^{1/3+3\varepsilon}$.

Our argument presented below is much more technical. One of the major issues we have ignored is that for an arbitrary modulus we need to partition summation in~\eqref{eq:S2heur} depending on the value of $(m,q)$. This modifies the length  and modulus of summation in a way which makes it difficult to take full advantage of the recursive inequality given by~\eqref{eq:S2finalH}. It is possible Theorem~\ref{thm:main1} may be improved for certain $q$ such as prime although we have not attempted to do so here.

The reader may be curious to what extent the argument from this paper may be extended to cover a range of parameters of the form
$$q^{1/2-\varepsilon}\le N \le q^{1/2+\varepsilon_1}.$$
The restriction to $N\le q^{1/2-\varepsilon}$ is used in two places. First  it allows us to ignore a term $N^3$ on the right hand side of~\eqref{eq:S4141414141}  which contributes $N^{3/4}$ in~\eqref{eq:weyl}. Second at~\eqref{eq:qN3lm} which implies that only the zero Fourier coefficient contributes after applying Poission summation to~\eqref{eq:S2b1heuristic2}. It would be interesting to see if one could use $\ell,s$ as in~\eqref{eq:ls1} obtained only from Dirichlet's theorem and a more careful application of Poission summation and stationary phase at~\eqref{eq:S2b1heuristic2} to improve the Weyl differencing bound for a longer range of parameters.

\section{Background on Fourier analysis}
We first collect some well known results and techniques from Fourier analysis. We will work with both the Fourier transform over $\R$ and the group of residues modulo an integer $q$. We use $\widehat f$ to denote the Fourier transform  of a function $f$ over the space where it is defined. In particular, for 
$$f:\R \rightarrow \C,$$
we have
$$\widehat f(\eta)=\int_{-\infty}^{\infty}f(x)e(-x\eta)dx,$$
and for
$$g:\Z/q\Z \rightarrow \C,$$
we have  
$$\widehat g(\eta)=\frac{1}{q}\sum_{x=1}^{q}g(x)e_q(-x\eta).$$
In this context, the  Poission summation formula states:
\begin{lemma}
\label{lem:ps}
For $f\in L^{1}(\R)$ smooth and $g$ as above, we have 
\begin{align*}
\sum_{n\in \Z}f(n)g(n)=\sum_{m\in \Z}\widehat f\left(\frac{m}{q}\right)\widehat g(-m).
\end{align*}
\end{lemma} 
The convolution of two functions $f_1,f_2$ is given by 
\begin{align*}
(f_1*f_2)(x)=\int_{\R}f_1(y)f_2(x-y)dy.
\end{align*}
\begin{lemma}
\label{lem:conv}
For $f_1,f_2\in L^{1}(\R)$ smooth we have  
\begin{align*}
\widehat{(f_1*f_2)}(x)=\widehat f_1(x)\widehat f_2(x),
\end{align*}
and 
\begin{align*}
\widehat{f_1f_2}(x)=(\widehat{f_1}*\widehat{f_2})(x).
\end{align*}
\end{lemma}
We will also use the Fourier inversion formula.
\begin{lemma}
\label{lem:fourierinversion}
For $f,\widehat f\in L^{1}(\R)$ smooth we have  
\begin{align*}
f(x)=\int_{\R}\widehat f(y)e(xy)dy.
\end{align*}
\end{lemma}
\begin{lemma}
\label{lem:partialsummation}
Let $N$ be an integer and $F,f$ complex valued functions on the interval $[1,N]$ with $f$ continuously differentiable and satisfying 
\begin{align}
\label{eq:fcondspartial}
f(x)\ll 1, \quad \text{and} \quad f'(x)\ll \frac{\Omega}{n}, \quad x\in [1,N],
\end{align}
for some $\Omega>0$. There exists $\gamma \in [0,1)$ such that 
\begin{align*}
\sum_{1\le n \le N}f(n)F(n)\ll (1+\Omega )N^{o(1)}\sum_{1\le n \le N}F(n)e(\gamma n).
\end{align*}
\end{lemma}
\begin{proof}
By partial summation and~\eqref{eq:fcondspartial}
\begin{align}
\label{eq:ps1}
\nonumber \sum_{1\le n \le N}f(n)F(n)&\ll \sum_{1\le n \le N}F(n)+\int_{1}^{N}f'(x)\left(\sum_{1\le  n \le x}F(n)\right)dx \\
&\ll \sum_{1\le n \le N}F(n)+\Omega N^{o(1)} \sum_{1\le n \le x}F(n),
\end{align}
for some $1\le x \le N$. We have 
\begin{align*}
\sum_{1\le n \le x}F(n)&=\int_{0}^{1}\sum_{1\le n \le N}F(n)\left(\sum_{1\le m \le x}e(\alpha(n-m)) \right)d\alpha \\
&\ll \int_{0}^{1}\left|\sum_{1\le m \le x}e(-\alpha m)\right|\left|\sum_{1\le n \le N}F(n)e(\alpha n)\right|d\alpha \\
& \ll \max_{\gamma \in [0,1)}\left|\sum_{1\le n \le N}F(n)e(\gamma n)\right|\int_{0}^{1}\left|\sum_{1\le m \le x}e(-\alpha m)\right|d\alpha,
\end{align*}
and the result follows from~\eqref{eq:ps1} after using 
\begin{align*}
\int_{0}^{1}\left|\sum_{1\le m \le x}e(-\alpha m)\right|d\alpha \ll \int_{0}^{1}\min\left\{N,\frac{1}{\|\alpha\|}\right\}d\alpha \ll N^{o(1)}.
\end{align*}
\end{proof} 
\begin{lemma}
\label{lem:apcomp}
Let $\ell, p,N$ be positive integers with $\ell,p\ll N^{O(1)}$. For any complex valued function $F$ on the interval $[1,N]$ there exists some $\gamma \in [0,1)$ such that 
\begin{align*}
\sum_{\substack{1\le \ell n+p\le N}}F(\ell n+p)\ll N^{o(1)}\sum_{1\le n \le N}F(n)e(\gamma n).
\end{align*}
\end{lemma}
\begin{proof}
We have 
\begin{align*}
\sum_{\substack{1\le \ell n+p\le N}}F(\ell n+p)&=\int_{0}^{1}\sum_{1\le n \le N}F(n)\left(\sum_{1\le \ell m +p \le N}e(\alpha(\ell m+p-n)) \right) \\
&\ll \int_{0}^{1}\left|\sum_{(1-p)/\ell\le m \le (N-p)/\ell}e(\alpha \ell m) \right|\left|\sum_{1\le n \le N}F(n)e(\alpha n) \right|d\alpha,
\end{align*}
and hence for some $\gamma \in [0,1)$
\begin{align*}
&\sum_{\substack{1\le \ell n+p\le N}}F(\ell n+p)\ll \left|\sum_{1\le n \le N}F(n)e(\gamma n) \right| \\ & \quad \quad \quad \quad \quad \quad \quad \quad \quad \times \int_{0}^{1}\left|\sum_{(1-p)/\ell\le m \le (N-p)/\ell}e(\alpha \ell m) \right|d\alpha,
\end{align*}
and the result follows after using 
\begin{align*}
\int_{0}^{1}\left|\sum_{(1-p)/\ell\le m \le (N-p)/\ell}e(\alpha \ell m) \right|d\alpha&\ll \frac{1}{\ell}\int_{0}^{\ell}\left|\sum_{(1-p)/\ell\le m \le (N-p)/\ell}e(\alpha m) \right|d\alpha \\ &  \ll N^{o(1)}.
\end{align*}
\end{proof}
Using a similar completion argument as above, it is possible to show the following, see~\cite[Lemma~2.2]{BI}.
\begin{lemma}
\label{lem:intervalcompletion}
Let $N,N_1,M,M_1$ be integers satisfying 
$$N_1\le M_1 \le M\le N.$$
For any complex valued function $F$ on the interval $[N_1,N]$, there exists some $\gamma \in \R$ such that 
\begin{align*}
\sum_{M_1\le n \le M}F(n)\ll N^{o(1)}\sum_{N_1\le n \le N}F(n)e(\gamma n).
\end{align*}
\end{lemma}
The following is known as a smooth partition of unity, see~\cite[Lemma~2]{Fou1}.
\begin{lemma}
\label{lem:smoothpart}
There exists a sequence $V_{0},V_1,\dots$ of smooth functions satisfying 
\begin{align}
\label{eq:smoothpart1}
\text{supp}(V_{\ell})\subseteq (2^{\ell-1},2^{\ell}],
\end{align}
\begin{align}
\label{eq:smoothpart2}
V_{\ell}^{(k)}(x)\ll \frac{1}{x^k}, \quad k\ge 1,
\end{align}
and for any $x\ge 1$ we have 
\begin{align}
\label{eq:smoothpart3}
\sum_{\ell\ge 0}V_{\ell}(x)=1.
\end{align}
\end{lemma}
Our next result follows by combining Lemma~\ref{lem:intervalcompletion}, Lemma~\ref{lem:smoothpart} and partitioning summation into dyadic intervals.
\begin{lemma}
\label{lem:intervalsmooth}
Let $F$ be a complex valued function on the integers. For any $N\gg 1$ there exists some $1\le K \le 2N,\gamma \in \R$ and smooth function $f$ satisfying 
\begin{align*}
\text{supp}(f)\subseteq [K,2K], \quad f^{(j)}(x)\ll \frac{1}{x^{j}},
\end{align*} 
such that 
\begin{align*}
\sum_{1\le n \le N}F(n)\ll N^{o(1)}\sum_{n\in \Z}f(n)F(n)e(\gamma n).
\end{align*}
\end{lemma}
\section{Gauss sums}
\label{sec:gauss}
Given integers $a,\ell,q,$ define the complete Gauss sum 
\begin{align}
\label{eq:Gdef}
G(a,\ell;q)=\sum_{\mu=1}^{q}e_q(a\mu^2+\ell \mu).
\end{align}
We next collect some results dating back to Gauss regarding the evaluation of $G(a,\ell;q)$.

 Given integers $p,q$ we let $(p/q)$ denote the Jacobi symbol. Recall quadratic reciprocity, see~\cite[Equation~1.2.7]{BEW}.
\begin{lemma}
\label{lem:qreciprocity}
For any odd positive integers $p,q$ we have 
\begin{align*}
\left(\frac{p}{q}\right)\left(\frac{q}{p}\right)=(-1)^{(p-1)(q-1)/4}.
\end{align*}
\end{lemma}
Evaluating the Jacobi symbol at even integers requires seperate treatement, see~\cite[pg. 17]{BEW}.
\begin{lemma}
\label{lem:jacobieven}
For any odd positive integer $q$ we have 
\begin{align*}
\left(\frac{2}{q}\right)=(-1)^{(q^2-1)/8}.
\end{align*}
\end{lemma}
 Our next result follows from completing the square in summation over $\mu$ in~\eqref{eq:Gdef}, see for example~\cite[Lemma~2.10]{BI}.
\begin{lemma}
\label{lem:gausssum1}
If $q\ge 1$ and $(a,q)=1$ then 
\begin{align*}
G(a,\ell,q)=\begin{cases} e_q\left(-\overline a\frac{\ell^2}{4}\right)G(a,0;q) \quad \text{if} \quad \ell \equiv 0 \mod{2}, \\ e_q\left(-\overline a\frac{\ell^2-1}{4}\right)G(\overline a,\overline a;q) \quad \text{if} \quad \ell\equiv 1 \mod{2}.  \end{cases}
\end{align*}
\end{lemma}
If $q$ is odd we may complete the square in a way which does not depend on the parity of $\ell$.
\begin{lemma}
\label{lem:gausssum2}
Let $q\ge 1$ be odd and $(a,q)=1$. For any integer $\ell$ 
\begin{align*}
G(a,\ell,q)=e_q(-\overline{4a}\ell^2)G(a,0;q).
\end{align*}
\end{lemma} 

We will require explicit evaluation of the Gauss sums occuring in Lemma~\ref{lem:gausssum1} and will consider the value of $q$ mod $4$ on a case by case basis. We first note a consequence of the Chinese Remainder Theorem.
\begin{lemma}
\label{lem:gaussmult}
For any $(q_1,q_2)=1$ and $a,\ell$ we have 
\begin{align*}
G(a,\ell;q_1q_2)=G(aq_1,\ell;q_2)G(aq_2,\ell,q_1).
\end{align*}
\end{lemma}
 The following is~\cite[Theorem~1.52]{BEW}.
\begin{lemma}
\label{lem:13}
Let $(a,q)=1$ with $q\ge 1$ and odd. Then 
\begin{align*}
G(a,0;q)= \begin{cases} \left(\frac{a}{q}\right)q^{1/2} \quad \text{if} \quad q\equiv 1 \mod{4}, \\
\left(\frac{a}{q}\right)iq^{1/2} \quad \text{if} \quad q\equiv 3 \mod{4}.
\end{cases}
\end{align*}
\end{lemma}
Using Lemma~\ref{lem:gaussmult} we may deal with the case $q\equiv 2 \mod{4}$.
\begin{lemma}
\label{lem:2}
Let $(a,2q)=1$ with $q\ge 1$ and odd. Then 
\begin{align*}
G(a,0;2q)= 0.
\end{align*}
\end{lemma}
\begin{proof}
By Lemma~\ref{lem:gaussmult}
\begin{align*}
G(a,0;2q)=G(aq,0;2)G(2a,0;q).
\end{align*}
Since $aq$ is odd, we see that 
\begin{align*}
G(aq,0;2)=0,
\end{align*}
which completes the proof.
\end{proof}
The following is~\cite[Theorem~1.5.4]{BEW}.
\begin{lemma}
\label{lem:4}
Let $(a,q)=1$  with $q\ge 1$ and $q\equiv 0 \mod{4}$. We have 
\begin{align*}
G(a,0;q)=\left(\frac{q}{a}\right)(1+i^{a})q^{1/2}.
\end{align*}
\end{lemma}
We next consider sums of the form $G(a,a;q)$. For odd $q$  these may be evaluated using Lemma~\ref{lem:gausssum2}. Our next result deals with the case of even $q$. 
\begin{lemma}
\label{lem:evenaa}
Let $q\ge 1$  be odd. For any $(a,2q)=1$ we have 
\begin{align*}
G(a,a;2q)= 2G(2a,a;q).
\end{align*}
\end{lemma}
\begin{proof}
By Lemma~\ref{lem:gaussmult}
\begin{align*}
G(a,a;2q)=G(2a,a;q)G(aq,a;2).
\end{align*}
Using that both $q,a$ are odd, we have 
\begin{align*}
G(aq,a;2)=\sum_{j=0}^{1}e^{\pi i(aqj^2+aj)}=2,
\end{align*}
from which the result follows.
\end{proof}
\begin{lemma}
\label{lem:evenaa1}
Let $q\ge 1$  be odd and $k\ge 2$. For any $(a,2q)=1$ and odd integer $b$ we have 
\begin{align*}
G(a,b;2^{k}q)=0.
\end{align*}
\end{lemma}
\begin{proof}
By Lemma~\ref{lem:gaussmult}
\begin{align*}
G(a,b;2^{k}q)=G(a2^{k},b;q)G(aq,b;2^{k}).
\end{align*}
Hence it is sufficient to show 
\begin{align*}
G(aq,b;2^{k})=0.
\end{align*}
We have 
\begin{align*}
G(aq,b;2^{k})=\sum_{y=1}^{2^{k}}e_{2^{k}}(aqy^2+by).
\end{align*}
Since $b$ is odd, the polynomial $F(x)=aqx^2+bx$ satisfies 
\begin{align*}
(2,F'(n))=1, \quad n\in \Z.
\end{align*}
By Hensel's Lemma, this implies for any integer $m$
\begin{align*}
|\{1\le n \le 2^{k} \ : \ F(n)\equiv m \mod{2^{k}}\}|\le 2.
\end{align*}
Since each $a,b,q$ are odd,
\begin{align*}
|\{1\le n \le 2^{k} \ : \ F(n)=m\}|=0 \quad \text{if $m$ odd}.
\end{align*}
By the pigeonhole principle 
\begin{align*}
|\{1\le n \le 2^{k} \ : \ F(n)=m\}|=\begin{cases} 2 \quad \text{if $m$ even,} \\ 0 \quad \text{otherwise}, \end{cases}
\end{align*}
and hence
\begin{align*}
G(aq,b;2^{k})=2\sum_{y=1}^{2^{k-1}}e_{2^{k-1}}(y)=0.
\end{align*}
\end{proof}

\section{Weyl differencing}
The main result of this section is Lemma~\ref{lem:weyliter} which will be used in an iterative manner to construct a sequence of good rational approximations to the amplitude function in the cubic exponential sums. The technique is based on Weyl differencing and lattice reduction.
%We next recall Weyl differencing for degree $3$ polynomials, see~\cite[Proposition~8.2]{IwKo} for details.
%\begin{lemma}
%\label{lem:weyldiff}
%Let $g$ be a cubic polynomial with real coefficients and leading coefficient $\alpha$. For any integer $N$ we have 
%\begin{align*}
%\left|\sum_{1\le n \le N}e(g(n)) \right|\le 2N\left(\frac{1}{N^3}\sum_{-N\le \ell_1,\ell_2\le N}\min\left(N,\| 6\alpha \ell_1 \ell_2\|^{-1}\right) \right)^{1/4}.
%\end{align*}
%\end{lemma}
%Lemma~\ref{lem:weyldiff} reduces bounding exponential sums to a lattice point counting problem. Assuming such sums are large we can use basis reduction to find a good approximation to $\alpha$ with small denominator. 

We first recall~\cite[Lemma~6.3]{DKSZ} which is stated for prime modulus only, the exact same proof works for arbitrary modulus. 
\begin{lemma}
\label{lem:shortpoint}
Let $q$ be a positive integer, $K \ge 1$ and $\cI,\cJ$  two intervals containing $h$ and $H$ integers respectivley.
For integer $b$ satisfying $(b,q)=1$, let $I(b)$ count the number of solutions to  the congruence 
\begin{align}
\label{eq:maincong}
yb\equiv x \mod{q}, \quad x\in \cI, \quad y\in \cJ.
\end{align}
One of the following two cases hold
\begin{align}
\label{eq:case11}
I(b)\ll \frac{Hh}{q}.
\end{align}
\begin{align}
\label{eq:case12}
\text{If} \ \   I(s)\ge &  K  \ \  \text{then there exists $\ell,s$ satisfying}  \nonumber  \\
&  b\equiv \ell \overline{s} \mod{q}, \quad  \ell \ll \frac{h}{K}, \quad |s|\ll \frac{H}{K}.
\end{align}
\end{lemma}
We next perform some preliminary manipulations and set up notation which will be used throughout the paper.
\begin{lemma}
\label{lem:weyl3}
Let $g$ be a cubic polynomial with real coefficients of the form 
\begin{align}
\label{eq:gpoly}
g(x)=\frac{ax^3}{q}+\gamma x,
\end{align}
with $(a,q)=1$.
  Let  $f$ be a smooth function satisfying 
\begin{align}
\label{eq:fconds1}
\text{supp}(f)\subseteq [N,2N], \quad f^{(j)}(x)\ll \frac{1}{x^{j}}.
\end{align}
Define
\begin{align}
\label{eq:q0delta1}
q_0=\frac{q}{(q,3)},
\end{align}
\begin{align}
\label{eq:bdef}
b=\begin{cases} 
3a \quad \text{if $(q,3)=1$} \\
a \quad \text{otherwise}.
\end{cases}
\end{align}
There exists some $d|q_0$ such that 
\begin{align*}
&\left|\sum_{n\in \Z}f(n)e(g(n)) \right|^2\ll N  \\ & \quad \quad \quad +q^{o(1)}\sum_{\substack{1\le |m|< dN/q_0 \\ (m,q_0)=1}}e\left(g\left(\frac{q_0m}{d}\right)\right)\sum_{n\in \Z}F_{d,m}(n)e(\rho_{q_0m/d}n)e_{d}(bmn^2).
\end{align*}
where 
\begin{align}
\label{eq:deltarhom}
 \rho_m= \frac{ b m^2}{q_0},
\end{align}
and
\begin{align}
\label{eq:Fdmdef}
F_{d,m}(n)=f(q_0m/d+n)f(n).
\end{align}
\end{lemma}
\begin{proof}
Let 
\begin{align*}
S=\sum_{n\in \Z}f(n)e(g(n)).
\end{align*}
Expand the square and apply the change of varialbe $m\rightarrow m+n$ to get
\begin{align*}
|S|^2&=\sum_{m,n\in \Z}f(m)f(n)e(g(m)-g(n)) \\ & =\sum_{m,n\in \Z}f(m+n)f(n)e(g(m+n)-g(n)).
\end{align*}
Using
\begin{align*}
g(m+n)-g(n)=g(m)+\frac{3am}{q} n^2+\frac{3a m^2}{q}n,
\end{align*}
and recalling~\eqref{eq:q0delta1},~\eqref{eq:bdef}
\begin{align*}
|S|^2=\sum_{m\in \Z}e(g(m))\sum_{n\in \Z}f(m+n)f(n)e_{q_0}(b m n^2+b m^2n).
\end{align*}
 By~\eqref{eq:fconds1}, if  $f(m+n)f(n)\neq 0$ then  $|m|\le N$ and hence 
\begin{align*}
|S|^2&=\sum_{|m|\le N}e(g(m))\sum_{n\in \Z}f(m+n)f(n)e_{q_0}(b m n^2+b m^2n) \\
&\ll N+\sum_{1\le |m|\le N}e(g(m))\sum_{n\in \Z}f(m+n)f(n)e_{q_0}(b m n^2+b m^2n).
\end{align*}
 We next partition summation over $m$ depending on the value of $(m,q_0)$ to get
\begin{align}
\label{eq:SSd}
S= \sum_{\substack{d|q_0}}S'_{d},
\end{align}
where 
\begin{align*}
S'_{d}&=\sum_{\substack{1\le |m|\le N \\ (m,q_0)=q_0/d }}e\left(g\left(m\right)\right)\sum_{n\in \Z}f(m+n)f(n)e_{q_0}(bm^2n)e_{q_0}(bmn^2) \\
&=\sum_{\substack{1 \le |m|\le dN/q_0 \\ (m,q_0)=1}}e\left(g\left(\frac{q_0m}{d}\right)\right)\sum_{n\in \Z}F_{d,m}(n)e(\rho_{q_0m/d}n)e_{d}(bmn^2).
\end{align*}
Taking a maximum over $d|q_0$ in~\eqref{eq:SSd} and using estimates for the divisor function we complete the proof.  
\end{proof}

%We will use Corollary~\ref{cor:weyl}  as the first step in an iteration to obtain a rational approximation to $a/q$ with very small denominator. The goal of the iteration is to relax the upper bounds on $s$ and $\ell$ in~\eqref{eq:bells}. Our next result gives the main input for this iteration. For later applications of Poission summation it will be convenient to work with summation over smooth functions.
\begin{lemma}
\label{lem:weyliter}
Let notation and conditions be as in Lemma~\ref{lem:weyl3} and assume that
\begin{align}
\label{eq:upassum}
\left|\sum_{n\in \Z}f(n)e(g(n)) \right| \ge \delta (qN)^{1/4+o(1)},
\end{align}
for some 
\begin{align}
\label{eq:deltaconds}
\frac{N^{1/2}}{q^{1/4}}\le \delta \le 1.
\end{align}
Let $Y\le N$ and define
\begin{align}
\label{eq:Sd1-1-1}
S_{d}=\sum_{\substack{dY/q_0\le |m|\le dN/q_0 \\ (m,q_0)=1}}e\left(g\left(\frac{q_0m}{d}\right)\right)\sum_{n\in \Z}F_{d,m}(n)e(\rho_{q_0m/d}n)e_{d}(bmn^2).
\end{align}
Suppose $N$ satisfies 
\begin{align}
\label{eq:Ncondsweyl11}
q^{1/3}\le N\ll \delta^2 q^{1/2}.
\end{align}
Either there exists integers $\ell_0, s_0$ satisfying 
\begin{align}
\label{eq:itercase1}
b\equiv \ell_0 \overline{s_0} \mod{d}, \quad |\ell_0|\ll \frac{1}{\delta^4}\left(\frac{d}{q}\right)^2\frac{Y}{N}, \quad |s_0|\ll \frac{1}{\delta^4}\left(\frac{d}{q}\right)^2\frac{NY^2}{q},
\end{align}
or 
\begin{align}
\label{eq:itercase2}
\left|\sum_{n\in \Z}f(n)e(g(n)) \right|^2\ll q^{o(1)}S_d.
\end{align}
\end{lemma}
\begin{proof}
Let 
\begin{align*}
S=\sum_{n\in \Z}f(n)e(g(n)).
\end{align*}
By Lemma~\ref{lem:weyl3} and the assumptions~\eqref{eq:upassum} and~\eqref{eq:deltaconds} we have 
\begin{align*}
|S|^2\ll q^{o(1)}\sum_{\substack{1\le |m|\le dN/q_0 \\ (m,q_0)=1}}e\left(g\left(\frac{q_0m}{d}\right)\right)\sum_{n\in \Z}F_{d,m}(n)e(\rho_{q_0m/d}n)e_{d}(bmn^2).
\end{align*}
With notation as in~\eqref{eq:Sd1-1-1}
\begin{align*}
|S|^2\ll q^{o(1)}(S'+S_d),
\end{align*}
where 
\begin{align*}
S'=\sum_{\substack{1\le |m|< dY/q_0 \\ (m,q_0)=1}}e\left(g\left(\frac{q_0m}{d}\right)\right)\sum_{n\in \Z}F_{d,m}(n)e(\rho_{q_0m/d}n)e_{d}(bmn^2).
\end{align*}
Either 
\begin{align}
\label{eq:itercase11}
|S|^2\ll q^{o(1)}S',
\end{align}
or 
\begin{align}
\label{eq:itercase21}
|S|^2\ll q^{o(1)}S_d.
\end{align}
If~\eqref{eq:itercase21} then we obtain~\eqref{eq:itercase2}. Suppose next~\eqref{eq:itercase11}. Applying the Cauchy-Schwarz inequality
\begin{align*}
|S|^4&\ll \frac{dY}{q}\sum_{|m|< dY/q_0}\left|\sum_{n\in \Z}F_{d,m}(n)e(\rho_{q_0m/d}n)e_{d}(bmn^2) \right|^2 \\
&\ll\frac{dY}{q}\sum_{|m|< dY/q_0}\sum_{n_1,n_2\in \Z}F_{d,m}(n_1)F_{d,m}(n_2) \\ & \quad \quad \quad \quad \quad \quad \quad \quad \quad \times e_{q_0}(b m (n_1^2-n_2^2)+\rho_{q_0m/d}(n_1-n_2)).
\end{align*}
Applying  the change of variable $n_1\rightarrow n_1+n_2$ gives 
\begin{align*}
|S|^4\ll \frac{dY}{q}\sum_{|m|< dY/q_0}\sum_{|n_1|\le 2N}\left|\sum_{n_2\in \Z}f_0(n_2)e_{d}(2bmn_1 n_2)\right|,
\end{align*}
where 
\begin{align*}
f_0(n_2)=F_{d,m}(n_1+n_2)F_{d,m}(n_2).
\end{align*}
By~\eqref{eq:fconds1} and partial summation over $n_2$ 
\begin{align*}
|S|^4\ll \frac{dY}{q}\sum_{|m|< dY/q_0}\sum_{|n_1|\le 2N}\min\left\{N,\frac{1}{\|2bmn_1/d\|}\right\}.
\end{align*}
Using~\eqref{eq:upassum}, estimates for the divisor function and isolating the contribution from either $m,n_1=0$, we arrive at 
\begin{align*}
\delta^4(qN)^{1+o(1)}\ll \frac{dYN^{o(1)}}{q}\sum_{\substack{|\ell| \ll dNY/q}}\min\left\{N,\frac{1}{\|b\ell /d\|}\right\}+\frac{dYN^2}{q}.
\end{align*}
Since $dY/q\ll N$, by~\eqref{eq:Ncondsweyl11} this simplifies to
\begin{align}
\label{eq:DD2}
\frac{\delta^4q^{2+o(1)}N}{dY}\ll N^{o(1)}\sum_{\substack{|\ell| \ll dNY/q}}\min\left\{N,\frac{1}{\|b\ell /d\|}\right\}.
\end{align}
For integer $j$ define the set 
\begin{align*}
\cL(j)=\left\{ |\ell|  \ll dNY/q  \ : \  \frac{2^{j-1}-1}{N}\le  \left\| b\ell /d \right\|< \frac{2^{j}-1}{N} \right\},
\end{align*}
so that~\eqref{eq:DD2} implies 
\begin{align*}
\frac{\delta^4q^{2+o(1)}N}{dY}\ll N^{1+o(1)}\sum_{j\ll \log{N}}\frac{|\cL(j)|}{2^{j}}\ll \frac{N^{1+o(1)}|\cL(j)|}{2^{j}},
\end{align*}
for some $1\le j \ll \log{N}$. With suitable $o(1)$ terms we obtain 
\begin{align}
\label{eq:Lub1}
\frac{\delta^4 2^{j}q^2}{dY}\ll |\cL(j)|.
\end{align}
Note that $|\cL(j)|$ is bounded by the number of solutions to the congruence 
\begin{align*}
bn\equiv m \mod{d}, \quad |m|\ll \frac{2^{j}d}{N}, \quad |n|\ll \frac{dNY}{q},
\end{align*}
to which we may apply Lemma~\ref{lem:shortpoint}. By~\eqref{eq:Ncondsweyl11} we may suppose 
\begin{align*}
|\cL(j)|\gg \frac{2^{j}dY}{q},
\end{align*}
which implies case~\eqref{eq:case12} of Lemma~\ref{lem:shortpoint}. Using~\eqref{eq:Lub1} we apply Lemma~\ref{lem:shortpoint}
with
\begin{align*}
K=\frac{c_0\delta^4 2^{j}q^2}{dY}, \quad h=\frac{c_12^{j}d}{N}, \quad H=\frac{c_2dNY}{q},
\end{align*}
for suitable absolute constants $c_0,c_1,c_2$. We obtain integers $\ell_0,s_0$ satisfying 
\begin{align*}
b\equiv \ell_0 \overline{s_0} \mod{d}, \quad |\ell_0|\ll \frac{1}{\delta^4}\left(\frac{d}{q}\right)^2\frac{Y}{N}, \quad |s_0|\ll \frac{1}{\delta^4}\left(\frac{d}{q}\right)^2\frac{NY^2}{q},
\end{align*}
which completes the proof.
\end{proof}
\section{Duality for summation over Gauss sums}
In this section we estimate the sums occuring in~\eqref{eq:Sd1-1-1} of Lemma~\ref{lem:weyliter}. The procedure consists of Poission summation, evaluation of Gauss sums, reciprocity for modular inverses, Poission summation, evaluation of Gauss sums then quadratic reciprocity which we split into a number of stages. The main result of this section is Lemma~\ref{lem:weyl6}.

\begin{lemma}
\label{lem:weyl4}
Let notation and conditions be as in Lemma~\ref{lem:weyliter} and suppose $\ell,s$ are any integers satisfying 
\begin{align}
\label{eq:ellsdef}
b\equiv \ell \overline{s} \mod{d}, \quad \ell>0.
\end{align}
Define $t\in \Z$ by
\begin{align}
\label{eq:tq011}
sb=\ell +td,
\end{align}
and
\begin{align}
\label{eq:Mjdef}
M_d=\max\left\{\left(\frac{|s|q_0}{\ell d^2Y}\right)^{1/2},\frac{N}{d}\right\}.
\end{align} 
Let $\varepsilon>0$ be small and suppose 
\begin{align}
\label{eq:Mjdconds}
\ell dM_d^2N^{11\varepsilon}\ll 1.
\end{align}
There exists polynomials $g^{*},g^{**}$ of the form 
\begin{align}
\label{eq:g*def}
g^{*}(x)=\frac{am^3}{4q}+\frac{b_1dm^3}{4\ell q_0}+\gamma_1m,
\end{align}
\begin{align}
\label{eq:g**def}
g^{**}(x)=\frac{am^3}{4q}+\frac{b_2 dm^3}{4\ell q_0}+\gamma_2m,
\end{align}
with  $b_1,b_2\in \Z,\gamma_1,\gamma_2\in \R$ and some $s_1|s$ such that defining $s_2$ by 
\begin{align}
\label{eq:s2def11}
s_1s_2=s,
\end{align}
 and $S^{(1)}_d,S^{(2)}_d$ by
\begin{align}
\label{eq:Sd1}
S^{(1)}_d=\sum_{\substack{dY/q_0s_1\le |m|\le dN/q_0s_1 \\ (m,q_0)=1 \\ (m,s)=1}}e\left(g^{*}\left(\frac{q_0s_1m}{d}\right)\right) \frac{G(bs_1m,0;d)G\left(s_2\overline d,-\frac{s_1tq_0m^2}{d};\ell m\right)}{m\ell},
\end{align}
\begin{align}
\label{eq:Sd2}
S^{(2)}_d=\sum_{\substack{dY/q_0s_1\le |m|\le dN/q_0s_1 \\ (m,q_0)=1 \\ (m,s)=1}}e\left(g^{**}\left(\frac{q_0s_1m}{d}\right)\right) \frac{G(\overline {bm},\overline{bm};q)G\left(s_2\overline d,s_2\overline{d}-\frac{s_1tq_0m^2}{d};\ell m\right)}{ m\ell},
\end{align}
 we have 
\begin{align*}
S_{d}\ll \left(\frac{dM_dN^{13\varepsilon}}{N}\right)^2\left(1+\frac{q |s|}{\ell  N^2Y}\right)(S^{(1)}_d+S^{(2)}_d).
\end{align*}
\end{lemma}
\begin{proof}
Recall from~\eqref{eq:Sd1-1-1} that
\begin{align*}
S_{d}=\sum_{\substack{dY/q_0\le |m|\le dN/q_0 \\ (m,q_0)=1}}e\left(g\left(\frac{q_0m}{d}\right)\right)\sum_{n\in \Z}F_{d,m}(n)e(\rho_{q_0m/d}n)e_{d}(bmn^2).
\end{align*}
With $s$ as in~\eqref{eq:ellsdef}, we partition summation over $m$ depending on the value of $(m,s)$ to get
\begin{align}
\label{eq:Sd1-1}
S_d=\sum_{s_1|s}S_{d,s_1},
\end{align} 
where 
\begin{align}
\label{eq:Sd1-}
S_{d,s_1}=\sum_{\substack{dY/q_0\le |m|\le dN/q_0 \\ (m,q_0)=1 \\ (m,s)=s_1}}e\left(g\left(\frac{q_0m}{d}\right)\right)\sum_{n\in \Z}F_{d,m}(n)e(\rho_{q_0m/d}n)e_{d}(bmn^2).
\end{align}
 Fix $s_1|s$ and consider $S_{d,s_1}$. We apply Lemma~\ref{lem:ps} to the  inner summation over $n$ to get
\begin{align*}
\sum_{n\in \Z}F_{d,m}(n)e_{d}(bmn^2)=\frac{1}{d}\sum_{n\in \Z}\widehat F_{d,m}\left(-\left(\frac{n}{d}+\rho_{q_0m/d}\right)\right)G(bm,n;d),
\end{align*}
after recalling the notation~\eqref{eq:Gdef}. Substituting the above into~\eqref{eq:Sd1-} gives 
\begin{align}
\label{eq:Sdspart}
S_{d,s_1}=S_{d,s_1,0}+S_{d,s_1,1},
\end{align}
where 
\begin{align*}
S_{d,s_1,i}&=\frac{1}{d}\sum_{\substack{dY/q_0\le |m|\le dN/q_0 \\ (m,q_0)=1 \\ (m,s)=s_1}}e\left(g\left(\frac{q_0m}{d}\right)\right) \\ & \quad \quad \quad \times \sum_{\substack{n\in \Z \\ n\equiv i \mod{2}}}\widehat F_{d,m}\left(-(\frac{n}{d}+\rho_{q_0m/d})\right)G(bm,n;d).
\end{align*}
We consider $S_{d,s_1,0}$ in detail then indicate necessary modifications for $S_{d,s_1,1}$. 
 By Lemma~\ref{lem:gausssum1}
\begin{align}
\label{eq:Sd077}
S_{d,s_1,0} & =\frac{1}{d}\sum_{\substack{dY/q_0\le |m|\le dN/q_0 \\ (m,q_0)=1 \\ (m,s)=s_1}}e\left(g\left(\frac{q_0m}{d}\right)\right)G(bm,0;d) \\ & \quad \quad \quad \quad \times \sum_{\substack{n\in \Z }}\widehat F_{d,m}\left(-\left(\frac{2n}{d}+\rho_{q_0m/d}\right)\right)e_d(-\overline{bm}n^2). \nonumber 
\end{align}
Recall $\ell,s$ satisfy~\eqref{eq:ellsdef}, so that
\begin{align*}
\frac{\overline{bm}n^2}{d}\equiv s \frac{\overline{\ell m}n^2}{d} \mod{1}.
\end{align*}
Using reciprocity for modular inverses 
\begin{align*}
\frac{\overline{\ell m}}{d}+\frac{\overline{d}}{\ell m}\equiv \frac{1}{\ell m d} \mod{1},
\end{align*}
the above implies
\begin{align}
\label{eq:modrec}
\frac{\overline{bm}n^2}{d}\equiv -\frac{s\overline{d}n^2}{\ell m}+\frac{sn^2}{\ell m d} \mod{1},
\end{align}
and hence 
\begin{align*}
&\sum_{\substack{n\in \Z }}\widehat F_{d,m}\left(-\left(\frac{2n}{d}+\rho_{q_0m/d}\right)\right)e_d(-\overline{bm}n^2) \\ & \quad \quad \quad \quad =\sum_{n\in \Z}\widehat F_{d,m}\left(-\left(\frac{2n}{d}+\rho_{q_0m/d}\right)\right)e\left(-\frac{sn^2}{\ell m d}\right)e_{\ell m}(s\overline d n^2).
\end{align*}
 Let $h$ be a smooth function satisfying 
\begin{align}
\label{eq:hdef1984}
\text{supp}(h)\subseteq [-2,2] \quad \text{and} \quad  h(x)=1 \quad \text{if} \quad |x|\le 1.
\end{align}
Recalling~\eqref{eq:Fdmdef}, we have 
\begin{align*}
\widehat F_{d,m}(x)=\int_{\R}f(q_0m/d+y)f(y)e(-yx).
\end{align*}
Repeated integration by parts shows that for any small $\varepsilon>0$ and large $C>0$ 
\begin{align}
\label{eq:FIparts}
\widehat F_{d,m}(x)\ll \frac{1}{N^{2C}x^2} \quad \text{provided} \quad  |x|\ge \frac{1}{N^{1-\varepsilon}}.
\end{align}
Since 
$$\widehat F_{d,m}(x)\ll N,$$
this implies
\begin{align*}
&\sum_{\substack{n\in \Z }}\widehat F_{d,m}\left(-\left(\frac{2n}{d}+\rho_{q_0m/d}\right)\right)e_d(-\overline{bm}n^2) \\ & =\sum_{|\frac{2n}{d}+\rho_{q_0m/d}|\le N^{-1+\varepsilon}}\widehat F_{d,m}\left(-\left(\frac{2n}{d}+\rho_{q_0m/d}\right)\right)e\left(-\frac{sn^2}{\ell m d}\right)e_{\ell m}(s\overline d n^2)+O\left(\frac{1}{N^{C}}\right).
\end{align*}
Using~\eqref{eq:hdef1984}
\begin{align*}
&\sum_{\substack{n\in \Z }}\widehat F_{d,m}\left(-\left(\frac{2n}{d}+\rho_{q_0m/d}\right)\right)e_d(-\overline{bm}n^2) = \\ & \sum_{|\frac{2n}{d}+\rho_{q_0m/d}|\le N^{-1+\varepsilon}}\widehat F_{d,m}\left(-\left(\frac{2n}{d}+\rho_{q_0m/d}\right)\right)  \\ & \quad \quad \quad \times h\left(\frac{N^{1-\varepsilon}}{d}(2n+d\rho_{q_0m/d}) \right)e\left(-\frac{sn^2}{\ell m d}\right)e_{\ell m}(s\overline d n^2)+O\left(\frac{1}{N^{C}}\right),
\end{align*}
and by another application of~\eqref{eq:FIparts}
\begin{align*}
&\sum_{\substack{n\in \Z }}\widehat F_{d,m}\left(-\left(\frac{2n}{d}+\rho_{q_0m/d}\right)\right)e_d(-\overline{bm}n^2)  \\ & \quad \quad \quad \quad =\sum_{n\in \Z}\widehat F_{d,m}\left(-\left(\frac{2n}{d}+\rho_{q_0m/d}\right)\right)h_m(n)e_{\ell m}(s\overline d n^2)+O\left(\frac{1}{N^{C}}\right),
\end{align*}
where 
\begin{align}
\label{eq:h0def}
h_m(x)=h\left(\frac{N^{1-\varepsilon}}{d}(2x+d\rho_{q_0m/d}) \right)e\left(-\frac{sx^2}{\ell m d}\right).
\end{align}
Define
\begin{align}
\label{eq:F1def}
F^{(1)}_{d,m}(x)=\widehat F_{d,m}\left(-\left(\frac{2x}{d}+\rho_{q_0m/d}\right)\right)h_m(x).
\end{align}
 By the above and~\eqref{eq:Sd077} 
\begin{align*}
 S_{d,s_1,0} &=\frac{1}{d}\sum_{\substack{dY/q_0\le |m|\le dN/q_0 \\ (m,q_0)=1 \\ (m,s)=s_1}}e\left(g\left(\frac{q_0m}{d}\right)\right)G(bm,0;d)\sum_{\substack{n\in \Z }}F^{(1)}_{d,m}(n)e_{\ell m}(s\overline d n^2) \\ & \quad \quad \quad \quad \quad +O\left(\frac{1}{N^{C-10}}\right) .
\end{align*}
Hence taking $C$ sufficiently large
\begin{align}
\label{eq:Sd0771} 
 S_{d,s_1,0}
&\ll\frac{1}{d}\sum_{\substack{dY/q_0s_1\le |m|\le dN/q_0s_1 \\ (m,q_0)=1 \\ (m,s)=1}}e\left(g\left(\frac{q_0s_1m}{d}\right)\right)G(bs_1m,0;d)\sum_{\substack{n\in \Z }}F^{(1)}_{d,s_1m}(n)e_{\ell m}(s_2\overline d n^2),
\end{align}
where $s_2$ is given by 
\begin{align}
\label{eq:s2def}
s_1s_2=s.
\end{align}
By Lemma~\ref{lem:ps}
\begin{align*}
\sum_{\substack{n\in \Z }}F^{(1)}_{d,s_1m}(n)e_{\ell m}(s_2\overline d n^2)=\frac{1}{\ell m}\sum_{n\in \Z}\widehat F^{(1)}_{d,s_1m}\left(-\frac{n}{\ell m}\right)G(s_2\overline d,n;\ell m).
\end{align*}
 Hence from~\eqref{eq:Sd0771}
\begin{align}
\label{eq:Sjds10-77}
\nonumber S_{d,s_1,0}&\ll \sum_{\substack{dY/q_0s_1\le |m|\le dN/q_0s_1 \\ (m,q_0)=1 \\ (m,s)=1}}e\left(g\left(\frac{q_0s_1m}{d}\right)\right)\frac{G(bs_1m,0;d)}{d\ell m} \\ & \quad \quad \quad \quad \times \sum_{n\in \Z}\widehat F^{(1)}_{d,s_1m}\left(-\frac{n}{\ell m}\right)G(s_2\overline d,n;\ell m).
% &=S^{(j)}_{d,s_1,0,0}+S^{(j)}_{d,s_1,0,1}+O\left(\frac{1}{N^{C}}\right),
\end{align}
%where 
%\begin{align*}
%S^{(j)}_{d,s_1,0,i}=&\frac{1}{d\ell m}\sum_{\substack{dY/q_0s_1\le |m|< qN/q_0s_1 \\ (m,q_0)=1 \\ (m,s)=1}}e\left(g\left(\frac{q_0s_1m}{d}\right)\right)G(as_1m,0;d) \\ & \quad \quad \quad \times \sum_{\substack{n\in \Z \\ n\equiv i \mod{2}}}\widehat F^{(1)}_{d,s_1m}\left(-\frac{n}{\ell m}\right)G(s_2\overline d,n;\ell m).
%\end{align*}
%We consider $S^{(j)}_{d,s_1,0,0}$ in detail then indicate neccessary modifications for $S^{(j)}_{d,s_1,0,1}$.  By Lemma~\ref{lem:gausssum1}
%\begin{align*}
%&\sum_{\substack{n\in \Z \\ n\equiv 0 \mod{2}}}\widehat F^{(1)}_{d,s_1m}\left(-\frac{n}{\ell m}\right)G(s_2\overline d,n;\ell m) \\ 
%& \quad =G(s_2\overline d,0;\ell m)\sum_{\substack{n\in \Z}}\widehat F^{(1)}_{d,s_1m}\left(-\frac{2n}{\ell m}\right)e_{\ell m}\left(-\overline{s_2}d n^2 \right),
%\end{align*}
%which implies
%\begin{align}
%\label{eq:Sjd00}
%S^{(j)}_{d,s_1,0,0}=&\sum_{\substack{dY/q_0s_1\le |m|< qN/q_0s_1 \\ (m,q_0)=1 \\ (m,s_1)=1}}e\left(g\left(\frac{q_0s_1m}{d}\right)\right)\frac{G(am,0;d)G(s_2\overline d,0;\ell m)}{md\ell} \\ & \quad \quad \quad \times \sum_{\substack{n\in \Z}}\widehat F^{(1)}_{d,s_1m}\left(-\frac{2n}{\ell m}\right)e_{\ell m}\left(-\overline{s_2}d n^2 \right). \nonumber
%\end{align}
Consider the inner summation over $n$. Recall~\eqref{eq:F1def} and write 

\begin{align*}
F^{(1)}_{d,s_1m}(x)=F_2(x)h_{s_1m}(x),
\end{align*}
with 
\begin{align}
\label{eq:F2def}
F_2(x)=\widehat F_{d,s_1m}\left(-\left(\frac{2x}{d}+\rho_{q_0s_1m/d}\right)\right).
\end{align}
By Lemma~\ref{lem:conv}
\begin{align}
\label{eq:F1hat123456789}
\widehat F^{(1)}_{d,s_1m}(x)=(\widehat F_2*\widehat h_{s_1m})(x)=\int_{\R}\widehat h_{s_1m}(y)\widehat F_2(x-y)dy,
\end{align}
and from~\eqref{eq:h0def} 
\begin{align*}
\nonumber \widehat h_{s_1m}(y)&=\int_{\R}h\left(\frac{N^{1-\varepsilon}}{d}(2z+d\rho_{q_0s_1m/d}) \right)e\left(-\beta_m z^2-yz\right)dz \\
&\ll \int_{\R}h\left(\frac{2N^{1-\varepsilon}z}{d}\right)e\left(-\beta_m z^2-(y-\beta_m d \rho_{q_0s_1m/d})z\right)dz,
\end{align*}
where 
\begin{align}
\label{eq:betadef}
\beta_m =\frac{s_2}{\ell m d}.
\end{align}
If  
$$|m|\ge\frac{dY}{q_0s_1} ,$$
 then 
\begin{align}
\label{eq:betambound}
|\beta_m|\ll \frac{|s|q_0}{\ell d^2Y}.
\end{align}
Repeated integration by parts shows that for any integer $k$
\begin{align*}
&\int_{\R}h\left(\frac{N^{1-\varepsilon}z}{d}\right)e\left(-\beta_m z^2-(y-\beta_m d \rho_{q_0m/d})z\right)dz \\ & \quad \quad \quad \quad \quad \ll \frac{1}{|y-\beta_m d \rho_{q_0m/d}|^k}\int_{\R}\left|\frac{d^{k}\left\{h\left(\frac{2N^{1-\varepsilon}z}{d}\right)e(-\beta_m z^2)\right\}}{dz^{k}} \right|dz,
\end{align*}
which combined with~\eqref{eq:hdef1984} implies
\begin{align*}
&\int_{\R}h\left(\frac{2N^{1-\varepsilon}z}{d}\right)e\left(-\beta_m z^2-(y-\beta_m d \rho_{q_0m/d})z\right)dz \\ & \quad \quad \quad \quad \ll \frac{1}{|y-\beta_m d \rho_{q_0m/d}|^k}\frac{d^{k+1}}{N^{(1-\varepsilon)(k+1)}}\left(|\beta_m|^{k}+\frac{N^{2k(1-\varepsilon)}}{d^{2k}} \right).
\end{align*}
Let $M_d$ be given by~\eqref{eq:Mjdef}. From~\eqref{eq:betambound} we obtain
\begin{align*}
&\int_{\R}h\left(\frac{N^{1-\varepsilon}z}{d}\right)e\left(-\beta_m z^2-(y-\beta_m d \rho_{q_0m/d})z\right)dz \\ & \quad \quad \quad \quad \ll \frac{N^{2k\varepsilon}}{|y-\beta_m d \rho_{q_0m/d}|^k}\frac{d}{N}\left(\frac{dM^2_{d}}{N}\right)^{k}.
\end{align*}
If 
\begin{align*}
|y-\beta_m d \rho_{q_0m/d}|\ge \frac{dM_d^2N^{10\varepsilon}}{N},
\end{align*}
then for any $C>0$, by choosing $k$ large enough in terms of $\varepsilon$, the above implies
\begin{align*}
\widehat h_{s_1m}(y)\ll \frac{1}{|y-\beta_m d \rho_{q_0m/d}|^2}\frac{1}{N^{2C}}.
\end{align*}
Substituting into~\eqref{eq:F1hat123456789} gives
\begin{align}
\widehat F^{(1)}_{d,s_1m}(x)&=\int_{|y|\le dM_d^2N^{10\varepsilon}/N}\widehat h_{s_1m}(y+\beta_m d \rho_{q_0s_1m/d})\widehat F_2(x-y-\beta_m d \rho_{q_0s_1m/d})dy \\ & \quad \quad \quad+O\left(\frac{1}{N^{C}}\right). \nonumber
\end{align} 
Combining the above with~\eqref{eq:Sjds10-77} and choosing $C$ sufficiently large 
\begin{align}
\label{eq:Sds11-56}
\nonumber & S_{d,s_1,0}\ll\int_{|y|\le dM_d^2N^{10\varepsilon}/N} \\ & \sum_{\substack{dY/q_0s_1\le |m|< qN/q_0s_1 \\ (m,q_0)=1 \\ (m,s)=1}}e\left(g\left(\frac{q_0s_1m}{d}\right)\right)\frac{G(b s_1 m,0;d)}{ md\ell}\widehat h_{s_1m}(y+\beta_m d \rho_{q_0s_1m/d}) \\ & \quad \quad \quad \times \sum_{\substack{n\in \Z}}\widehat F_2\left(\frac{n}{\ell m}-y-\beta_m d \rho_{q_0s_1m/d}\right) G(s_2\overline d,-n;\ell m) dy. \nonumber
\end{align}
Recalling~\eqref{eq:h0def}, we have 
\begin{align*}
& \widehat h_{s_1m}(y+\beta_m d \rho_{q_0s_1m/d})= \int_{\R}h\left(\frac{2zN^{1-\varepsilon}}{d}\right) \\ & \quad \quad \quad \times e\left(-\beta_m(z-d\rho_{q_0s_1m/d}/2)^2-(y+\beta_m d \rho_{q_0s_1m/d})(z-d\rho_{q_0s_1 m/d}/2)\right)dz.
\end{align*}
Substituting the above into~\eqref{eq:Sds11-56}, taking a maximum over $y,z$ and using that $h$ is supported in $[-2,2]$ we get
\begin{align}
\label{eq:Sjds10-76}
& \nonumber S_{d,s_1,0}\ll \left(\frac{dM_dN^{10\varepsilon}}{N}\right)^2 \\ & \quad \quad \quad \quad \times \sum_{\substack{dY/q_0s_1\le |m|\le dN/q_0s_1 \\ (m,q_0)=1 \\ (m,s)=1}}e\left(g\left(\frac{q_0s_1m}{d}\right)+g_0(m)\right)\frac{G(bs_1m,0;d)}{md\ell} \\ 
& \quad \quad \times \sum_{\substack{n\in \Z}}\widehat F_2\left(\frac{n}{\ell m}-y-\beta_m d \rho_{q_0s_1m/d}\right) G(s_2\overline d,-n;\ell m), \nonumber
\end{align}
for some $y,z$ satisfying 
\begin{align}
\label{eq:yzconds}
 |y|\le \frac{dM_d^2N^{10\varepsilon}}{N} \quad |z|\le \frac{dN^{\varepsilon}}{N},
\end{align}
and 
\begin{align}
\label{eq:g0def}
g_0(m)=-\beta_m(z-d\rho_{q_0s_1m/d}/2)^2-(y+\beta_m d \rho_{q_0ms_1/d})(z-d\rho_{q_0 ms_1/d}/2).
\end{align}
We next simplify summation over $n$. Recalling~\eqref{eq:F2def}, for any real number $y$ we have 
\begin{align*}
\widehat F_2\left(2y/d\right)&=\int_{\R}\widehat F_{d,s_1m}\left(-\left(\frac{2x}{d}+\rho_{q_0s_1m/d}\right)\right)e(-2xy/d)dx \\
&=\frac{d}{2}e(\rho_{q_0s_1m/d}y)\int_{\R}\widehat F_{d,s_1m}(u)e\left(uy \right)du.
\end{align*}
Hence by Lemma~\ref{lem:fourierinversion}
\begin{align}
\label{eq:Finv-F2}
\widehat F_2\left(2y/d\right)=\frac{d}{2}e(\rho_{q_0s_1m/d}y)F_{d,s_1m}(y).
\end{align}
Substituting~\eqref{eq:Finv-F2} into~\eqref{eq:Sjds10-76} gives
\begin{align}
\label{eq:above77}
\nonumber & \sum_{\substack{n\in \Z}}\widehat F_2\left(\frac{n}{\ell m}-y-\beta_m d \rho_{q_0s_1m/d}\right) G(s_2\overline d,-n;\ell m) \\ & \quad \quad \quad = \frac{d}{2}e\left(-\rho_{q_0s_1m/d}\left(\frac{dy}{2}+\frac{\beta_m d^2 \rho_{q_0s_1m/d}}{2}\right)\right) \\ & \quad \quad \times \sum_{\substack{n\in \Z}}F_{d,s_1m}\left(\frac{dn}{2\ell m}-\frac{dy}{2}-\frac{d^2 \beta_m \rho_{q_0 s_1m/d}}{2} \right)G(s_2\overline d,-n;\ell m)e\left(\rho_{q_0s_1m/d}\frac{dn}{2\ell m}\right). \nonumber
\end{align}
We next show only one value of $n$ contributes to summation in~\eqref{eq:above77}. Recalling~\eqref{eq:fconds1} and~\eqref{eq:Fdmdef}  if 
$$F_{d,s_1m}\left(\frac{dn}{2\ell m}-\frac{dy}{2}-\frac{d^2 \beta_m \rho_{q_0 m/d}}{2} \right)\neq 0,$$
then 
\begin{align}
\label{eq:sumconditions}
N\le \frac{dn}{2\ell m}-\frac{dy}{2}-\frac{d^2 \beta_m \rho_{q_0 s_1m/d}}{2}\le 2N.
\end{align}
By~\eqref{eq:yzconds}
\begin{align*}
dy\ll \frac{d^2M_d^2N^{10\varepsilon}}{N},
\end{align*}
so that if $n$ satisfies~\eqref{eq:sumconditions} then
\begin{align*}
\left|\frac{dn}{\ell m}-d^2 \beta_m \rho_{q_0s_1m/d}\right|\ll N+\frac{d^2M_d^2N^{10\varepsilon}}{N}\ll \frac{d^2M_d^2N^{10\varepsilon}}{N}.
\end{align*}
By~\eqref{eq:deltarhom},~\eqref{eq:s2def},~\eqref{eq:betadef} and using that $m\ll N$

\begin{align}
\label{eq:2nsteplast}
\left|n-\frac{s_1sbq_0 m^2}{d^2}\right|\ll \ell dM_d^2N^{10\varepsilon}.
\end{align}
Recalling~\eqref{eq:tq011}, for some integer $t$ we have  
\begin{align*}
sb=\ell +td,
\end{align*}
which substituted into~\eqref{eq:2nsteplast} 
 gives
\begin{align}
\label{eq:2nstep1}
\left|n-\frac{s_1tq_0}{d}m^2\right|\ll \ell \left(dM_d^2N^{10\varepsilon}+\frac{s_1q_0 m^2}{d^2} \right).
\end{align} 
Using that $m\ll dN/q_0s_1,$ we have 
\begin{align*}
\frac{s_1q_0 m^2}{d^2}\ll \frac{N^2}{q}\ll dM_d^2N^{10\varepsilon},
\end{align*}
which simplifies~\eqref{eq:2nstep1} to
\begin{align*}
\left|n-\frac{s_1tq_0}{d}m^2\right|\ll \ell dM_d^2N^{10\varepsilon}.
\end{align*} 
By~\eqref{eq:Mjdconds}, the only term which contributes to summation in~\eqref{eq:above77}  is
\begin{align*}
n=\frac{s_1tq_0}{d}m^2.
\end{align*}
This implies that
\begin{align*}
 & \sum_{\substack{n\in \Z}}\widehat F_2\left(\frac{n}{\ell m}-y-\beta_m d \rho_{q_0s_1m/d}\right) G(s_2\overline d,-n;\ell m) \\ & \quad \quad \quad = \frac{d}{2}e\left(-\rho_{q_0s_1m/d}\left(\frac{dy}{2}+\frac{\beta_m d^2 \rho_{q_0s_1m/d}}{2}\right)\right)F_{d,s_1m}\left(-\frac{ s_1q_0 m}{2d}-\frac{dy}{2} \right) \\ & \quad \quad \quad \quad \quad \times G\left(s_2\overline d,-\frac{s_1tq_0m^2}{d};\ell m\right)e\left(\rho_{q_0s_1m/d}\frac{s_1 t q_0 m}{2\ell }\right). 
\end{align*}
We subsitute the above into~\eqref{eq:Sjds10-76} then simplify. This gives 
\begin{align*}
& S_{d,s_1,0}\ll \left(\frac{dM_dN^{10\varepsilon}}{N}\right)^2 \\ &  \times \sum_{\substack{dY/q_0s_1\le |m|\le dN/q_0s_1 \\ (m,q_0)=1 \\ (m,s)=1}}F_{d,s_1m}\left(-\frac{s_1q_0 m}{2d}-\frac{dy}{2} \right)e\left(g\left(\frac{q_0s_1m}{d}\right)+g_0(m)+g_1(m)\right) \\ & \quad \quad \quad \quad \quad \times \frac{G(bs_1m,0;d)G\left(s_2\overline d,-\frac{s_1tq_0m^2}{d};\ell m\right)}{m\ell},
\end{align*}
where 
\begin{align*}
g_1(m)=-\rho_{q_0s_1m/d}\left(\frac{dy}{2}+\frac{\beta_m d^2 \rho_{q_0s_1m/d}}{2}\right)+\rho_{q_0s_1m/d}\frac{s_1 t q_0 m}{2\ell }.
\end{align*}
Recalling~\eqref{eq:deltarhom},~\eqref{eq:betadef} and~\eqref{eq:g0def} 
\begin{align*}
g_0(m)+g_1(m)&=-\beta_m(z-d\rho_{q_0s_1m/d}/2)^2-(y+\beta_m d \rho_{q_0ms_1/d})(z-d\rho_{q_0 ms_1/d}/2) \\ 
& -\rho_{q_0s_1m/d}\left(\frac{dy}{2}+\frac{\beta_m d^2 \rho_{q_0s_1m/d}}{2}\right)+\rho_{q_0s_1m/d}\frac{s_1 t q_0 m}{2\ell } \\
&=-\beta_mz^2-\frac{\beta_md^2\rho_{q_0s_1m/d}^2}{4}+\frac{\rho_{q_0s_1m/d}s_1tq_0 m}{2\ell}-yz \\
&=-\frac{z^2s_2}{\ell m d}-\frac{sb^2}{4\ell q_0}\left(\frac{q_0 s_1m}{d}\right)^3+\frac{b t d}{2\ell q_0}\left(\frac{q_0 s_1 m}{d}\right)^3-yz.
\end{align*}
 Recalling~\eqref{eq:gpoly},~\eqref{eq:q0delta1},~\eqref{eq:bdef} and~\eqref{eq:tq011} we have 
\begin{align*}
& g\left(\frac{q_0s_1m}{d}\right)+g_0(m)+g_1(m) \\ &=\frac{a}{4q}\left(\frac{q_0s_1m}{d}\right)^3+\frac{b_1d}{4\ell q}\left(\frac{q_0 s_1 m}{d}\right)^3+\gamma \frac{q_0s_1m}{d}-\frac{z^2s_2}{\ell m d},
\end{align*}
for some $b_1\in \Z$. Hence with $g^{*}$ given by
\begin{align*}
g^{*}(x)=\frac{am^3}{4q}+\frac{b_1d m^3}{4\ell q_0}+\gamma m,
\end{align*}
we have 
\begin{align*}
& S_{d,s_1,0}\ll \left(\frac{dM_dN^{10\varepsilon}}{N}\right)^2 \\ &  \times \sum_{\substack{dY/q_0s_1\le |m|\le dN/q_0s_1 \\ (m,q_0)=1 \\ (m,s)=1}}F_{d,s_1m}\left(-\frac{s_1q_0 m}{2d}-\frac{dy}{2} \right)e\left(g^{*}\left(\frac{q_0s_1m}{d}\right)-\frac{z^2s_2}{\ell m d}\right) \\ & \quad \quad \quad \quad \quad \times \frac{G(bs_1m,0;d)G\left(s_2\overline d,-\frac{s_1tq_0m^2}{d};\ell m\right)}{m\ell}.
\end{align*}
Our last step is to remove the terms $F_{d,s_1m}$ and $e(-z^2s_2/(\ell m d))$ using partial summation. Recalling~\eqref{eq:Fdmdef}
\begin{align*}
F_{d,s_1m}\left(-\frac{s_1q_0 m}{2d}-\frac{dy}{2}\right)=f\left(\frac{q_0s_1m}{2d}-\frac{dy}{2}\right)f\left(-\frac{q_0 s_1 m}{2d}-\frac{dy}{2}\right),
\end{align*}
hence from~\eqref{eq:fconds1} if $m$ satisfies 
\begin{align*}
dY/q_0s_1\le |m|\le dN/q_0s_1,
\end{align*}
then 
\begin{align*}
\frac{dF_{d,s_1m}\left(-\frac{q_0 s_1 m}{2d}-\frac{dy}{2}\right)}{dm}\ll \frac{1}{m}.
\end{align*}
By~\eqref{eq:yzconds}, for $m$ satisfying 
\begin{align*}
|m|\ge dY/q_0s_1,
\end{align*}
we have 
\begin{align*}
\frac{d e\left(-\frac{z^2s_2}{\ell m d}\right)}{dm}\ll \frac{N^{2\varepsilon}q_0 |s|}{\ell N^2Y}\frac{1}{m}.
\end{align*}
Hence by Lemma~\ref{lem:partialsummation}, for some $\gamma''\in [0,1]$ we have 
\begin{align}
\label{eq:Sds10-final}
\nonumber & S_{d,s_1,0}\ll \left(\frac{dM_dN^{13\varepsilon}}{N}\right)^2\left(1+\frac{q_0 |s|}{\ell  N^2Y}\right) \\ &  \sum_{\substack{dY/q_0s_1\le |m|\le dN/q_0s_1 \\ (m,q_0)=1 \\ (m,s)=1}}e\left(g^{*}\left(\frac{q_0s_1m}{d}\right)+\gamma'' m\right) \frac{G(bs_1m,0;d)G\left(s_2\overline d,-\frac{s_1tq_0m^2}{d};\ell m\right)}{m\ell}.
\end{align}
Returning to~\eqref{eq:Sdspart}, we next indicate the neccessary modifications to the above argument to estimate $S_{d,s_1,1}$. We have 
\begin{align*}
S_{d,s_1,1}&=\frac{1}{d}\sum_{\substack{dY/q_0\le |m|\le dN/q_0 \\ (m,q_0)=1 \\ (m,s)=s_1}}e\left(g\left(\frac{q_0m}{d}\right)\right) \\ & \quad \quad \quad \sum_{\substack{n\in \Z}}\widehat F_{d,m}\left(-\left(\frac{2n+1}{d}+\rho_{q_0m/d}\right)\right)G(bm,2n+1;d),
\end{align*}
and hence by Lemma~\ref{lem:gausssum1}
\begin{align*}
S_{d,s_1,1}&=\frac{1}{d}\sum_{\substack{dY/q_0\le |m|\le dN/q_0 \\ (m,q_0)=1 \\ (m,s)=s_1}}e\left(g\left(\frac{q_0m}{d}\right)\right)G(\overline {bm},\overline{bm};q) \\ & \quad \quad \quad \sum_{\substack{n\in \Z}}\widehat F_{d,m}\left(-\left(\frac{2n+1}{d}+\rho_{q_0m/d}\right)\right)e_d(-\overline{bm}(n^2+n)).
\end{align*}
Using reciprocity as in~\eqref{eq:modrec} gives
\begin{align*}
&\sum_{\substack{n\in \Z}}\widehat F_{d,m}\left(-\left(\frac{2n+1}{d}+\rho_{q_0m/d}\right)\right)e_d(-\overline{bm}(n^2+n))= \\ & \quad \quad   \sum_{\substack{n\in \Z}}\widehat F_{d,m}\left(-\left(\frac{2n+1}{d}+\rho_{q_0m/d}\right)\right)e_{\ell m}(s\overline{d}(n^2+n))e\left( -\frac{s (n^2+n)}{\ell m d}\right).
\end{align*}
Following the argument as in the case $S_{d,s_1,0}$ with changes to~\eqref{eq:h0def} and~\eqref{eq:F1def} given by
$$
h_m(x)=h\left(\frac{N^{1-\varepsilon}}{d}(2x+1+d\rho_{q_0m/d}) \right)e\left(-\frac{s(x^2+x)}{\ell m d}\right),
$$
$$
F^{(1)}_{d,m}(x)=\widehat F_{d,m}\left(-\left(\frac{2x+1}{d}+\rho_{q_0m/d}\right)\right)h_0(x),
$$
we arrive at an analouge of~\eqref{eq:Sds11-56}
\begin{align*}
& S_{d,s_1,1}\ll\int_{|y|\le dM_d^2N^{10\varepsilon}/N} \\ & \sum_{\substack{dY/q_0s_1\le |m|< qN/q_0s_1 \\ (m,q_0)=1 \\ (m,s)=1}}e\left(g\left(\frac{q_0s_1m}{d}\right)\right)\frac{G(\overline{bs_1m},\overline{bs_1m};d)}{md\ell}\widehat h_{s_1m}(y+\beta_m d \rho_{q_0s_1m/d}) \\ & \quad \quad \quad \times \sum_{\substack{n\in \Z}}\widehat F_2\left(\frac{n}{\ell m}-y-\beta_m d \rho_{q_0s_1m/d}\right) G(s_2\overline d,s_2\overline{d}-n;\ell m) dy,
\end{align*}
where 
\begin{align*}
F_2(x)=\widehat F_{d,s_1m}\left(-\left(\frac{2x+1}{d}+\rho_{q_0s_1m/d}\right)\right).
\end{align*}
This modifies~\eqref{eq:Finv-F2} to 
\begin{align*}
\widehat F_2\left(\frac{2y}{d}\right)=\frac{d}{2}e\left(\left(\rho_{q_0s_1m/d}+\frac{1}{d}\right)y \right)F_{d,s_1m}(y),
\end{align*}
and the rest of the proof is similar to before. Our variant of~\eqref{eq:Sds10-final} becomes 
\begin{align}
\label{eq:Sds11-final}
\nonumber & S_{d,s_1,1}\ll \left(\frac{dM_dN^{13\varepsilon}}{N}\right)^2\left(1 +\frac{q_0 s}{\ell  N^2Y}\right) \\ &  \sum_{\substack{dY/q_0s_1\le |m|\le dN/q_0s_1 \\ (m,q_0)=1 \\ (m,s)=1}}e\left(g^{**}\left(\frac{q_0s_1m}{d}\right)+\gamma''' m\right) \frac{G(\overline {bm},\overline{bm};q)G\left(s_2\overline d,s_2\overline{d}-\frac{s_1tq_0m^2}{d};\ell m\right)}{m\ell},
\end{align}
where $g^{**}$ is a polynomial of the form 
\begin{align*}
g^{**}(x)=\frac{am^3}{4q}+\frac{b_2 dm^3}{4\ell q_0}+\gamma_2 m,
\end{align*}
 for suitable $b_2 \in \Z$ and $\gamma_2\in \R$.

Combining~\eqref{eq:Sds10-final},~\eqref{eq:Sds11-final} and~\eqref{eq:Sdspart} then taking a maximum over $s_1|s$ in~\eqref{eq:Sd1-1} we complete the proof, after suitable renaming.
\end{proof}
Our next step is to use results from Section~\ref{sec:gauss} to simplify the sums $S_d^{(i)}$ occuring in Lemma~\ref{lem:weyl4}. We first consider $S^{(1)}_d$.
\begin{lemma}
\label{lem:weylS1}
With notation and conditions as in Lemma~\ref{lem:weyl4}, there exists a positive integer $j$ and a real number $\gamma_0$ such that
\begin{align*}
& S^{(1)}_{d}\ll \frac{N^{o(1)}(d\ell)^{1/2}}{2^{j/2}}\sum_{\substack{dY/2^{j}q_0s_1\le m\le dN/2^{j}q_0s_1 \\ (m,sq_0)=1  }}\frac{1}{m^{1/2}}e\left(\frac{1}{4}g\left(\frac{2^{j}q_0s_1m}{d}\right)+\gamma_0 m\right).
\end{align*}

\end{lemma}
\begin{proof}
First recall~\eqref{eq:Sd1}
\begin{align*}
\nonumber S^{(1)}_d&=\sum_{\substack{dY/q_0s_1\le |m|\le dN/q_0s_1 \\ (m,q_0)=1 \\ (m,s)=1}}e\left(g^{*}\left(\frac{q_0s_1m}{d}\right)\right) \frac{G(bs_1m,0;d)G\left(s_2\overline d,-\frac{s_1tq_0m^2}{d};\ell m\right)}{m\ell}.
\end{align*}
We partition summation depending on the sign and  largest power of $2$ dividing $m$ to get
\begin{align}
\label{eq:S1dpart-2}
S^{(1)}_d\ll\sum_{j\ll \log{N}}S^{(1)}_{d,j,+}+S^{(1)}_{d,j,-},
\end{align}
where
\begin{align}
\label{eq:S1djpart-21}
\nonumber & S^{(1)}_{d,j,\pm}=\sum_{\substack{dY/q_0s_1\le m\le dN/q_0s_1 \\ (m,q_0)=1 \\ (m,s)=1 \\ 2^{j}||m}}e\left(g^{*}\left(\pm\frac{q_0s_1m}{d}\right)\right) \frac{G(\pm bs_1m,0;d)G\left(s_2\overline d,-\frac{s_1tq_0m^2}{d};\pm\ell m\right)}{m\ell} \\
&=\sum_{\substack{dY/2^{j}q_0s_1\le m\le dN/2^{j}q_0s_1 \\ (2^{j}m,sq_0)=1 \\ m\equiv 1 \mod{2} }}e\left(g^{*}\left(\pm\frac{2^{j}q_0s_1m}{d}\right)\right) \\ & \quad \quad \quad \quad \quad \quad \quad \times \frac{G( \pm 2^jbs_1m,0;d)G\left(s_2\overline d,-\frac{s_1tq_02^{2j}m^2}{d};\pm\ell 2^{j}m\right)}{2^{j}m\ell}. \nonumber
\end{align}
Fix some $j$ and consider $S^{(1)}_{d,j,\pm}$. We provide details for $S^{(1)}_{d,j,+},$ after taking complex conjugates a similar argument applies to $S^{(1)}_{d,j,-}$.    For the right hand side of~\eqref{eq:S1djpart-21} to be nonzero, by Lemma~\ref{lem:2} we must have either 
\begin{align*}
d\equiv 0 \mod{4}, \quad d\equiv 1 \mod{4} \quad \text{or} \quad d\equiv 3 \mod{4}.
\end{align*}
If $d\equiv 0 \mod{4}$ then for summation conditions in~\eqref{eq:S1djpart-21} to be nonempty we must have $j=0$.  Hence by Lemma~\ref{lem:4}
\begin{align}
\label{eq:d4}
G(bs_1m,0;d)=\left(\frac{d}{bs_1 m}\right)\left(1+i^{bs_1 m}\right)d^{1/2}.
\end{align}
If either $d\equiv 1 \mod{4}$ or $d\equiv 3 \mod{4}$  then by Lemma~\ref{lem:qreciprocity} and Lemma~\ref{lem:13}
\begin{align}
\label{eq:d1}
\nonumber G(bs_1m,0;d)&=\varepsilon_d\left(\frac{2^{j}bs_1 m}{d}\right)d^{1/2} \\ & =\varepsilon_d\left(\frac{2^{j}bs_1}{d}\right) \left(\frac{d}{m}\right)(-1)^{(d-1)(m-1)/4}d^{1/2},
\end{align}
where 
\begin{align}
\label{eq:epsddef}
\varepsilon_d=\begin{cases}
1 \quad \text{if} \quad d\equiv 1 \mod{4} \\
i \quad \text{if} \quad d\equiv 3 \mod{4}.
\end{cases}
\end{align}
In either case~\eqref{eq:d4} or~\eqref{eq:d1}, we see that there exists some $\gamma_0\in \R$ such that 
\begin{align}
\label{eq:S1dj-990} 
S^{(1)}_{d,j,+}&\ll d^{1/2}\sum_{\substack{dY/2^{j}q_0s_1\le m\le dN/2^{j}q_0s_1 \\ (2^{j}m,sq_0)=1 \\ m\equiv 1 \mod{2} }}e\left(g^{*}\left(\frac{2^{j}q_0s_1m}{d}\right)+\gamma_0 m\right) \\ & \quad \quad \quad \quad \quad \quad \quad \quad \quad \quad \quad  \times \left(\frac{d}{m}\right)\frac{G\left(s_2\overline d,-\frac{s_1tq_02^{2j}m^2}{d};\ell 2^{j}m\right)}{2^{j}m\ell}. \nonumber
\end{align}
Fix $m$ satisfying conditions of summation in~\eqref{eq:S1dj-990} and consider the term 
\begin{align*}
G\left(s_2\overline d,-\frac{s_1tq_02^{2j}m^2}{d};\ell 2^{j}m\right).
\end{align*}
We proceed on a case by case basis depending on whether $j=0$ or not. If $j\ge 1$ then by Lemma~\ref{lem:gausssum1}
\begin{align*}
G\left(s_2\overline d,-\frac{s_1tq_02^{2j}m^2}{d}; 2^{j}\ell m\right)= e_{2^{j}\ell m}\left( \frac{\overline{s_2}s^2_1t^2q^2_02^{4j-2}m^4}{d} \right)G( s_2 \overline{d},0;2^{j}\ell  m),
\end{align*}
and hence there exists an integer $b_3$ such that 
\begin{align*}
G\left( s_2\overline d,-\frac{s_1tq_02^{2j}m^2}{d}; 2^{j}\ell  m\right)= e_{\ell}\left( b_3 m^3 \right)G( s_2 \overline{d},0;2^{j}\ell  m).
\end{align*}
Suppose next $j=0$. If
$$\ell m \equiv 2 \mod{4} \quad \text{and} \quad \frac{s_1tq_0m^2}{d}\equiv 0 \mod{2},$$ then by Lemma~\ref{lem:gausssum1} and  Lemma~\ref{lem:2}
\begin{align*}
G\left( s_2\overline d,-\frac{s_1tq_02^{2j}m^2}{d}; 2^{j}\ell  m\right)=0.
\end{align*}
If
$$\ell  m\equiv 2 \mod{4} \quad \text{and} \quad \frac{s_1tq_02^{2j}m^2}{d}\equiv 1 \mod{2},$$
 then writing $\ell=2\ell_0$, we have by Lemma~\ref{lem:gausssum2} and Lemma~\ref{lem:gaussmult}
\begin{align*}
G\left( s_2\overline d,-\frac{s_1tq_02^{2j}m^2}{d}; 2^{j}\ell  m\right)&=2G\left( 2s_2\overline d,-\frac{s_1tq_02^{2j}m^2}{d}; \ell_0  m\right) \\
&=2e_{\ell}(b_3'm^3)G\left( 2s_2\overline d,0; \ell_0  m\right),
\end{align*}
for some integer $b_3'$. Finally, if $j=0$ and 
$$\ell m \equiv 1 \mod{2},$$
then by Lemma~\ref{lem:gausssum2}
\begin{align*}
G\left( s_2\overline d,-\frac{s_1tq_02^{2j}m^2}{d}; 2^{j}\ell  m\right)=e_{\ell}(b^{''}_3m^3)G\left( 2s_2\overline d,0; \ell  m\right),
\end{align*}
for some integer $b_3''$. Combining the above with Lemma~\ref{lem:13}, Lemma~\ref{lem:4},~\eqref{eq:S1dj-990} and using that 
$$\left(\frac{d}{m}\right)^2=1,$$
we get 
\begin{align*}
S^{(1)}_{d,j,+}&\ll d^{1/2}\sum_{\substack{dY/2^{j}q_0s_1\le m\le dN/2^{j}q_0s_1 \\ (m,sq_0)=1 \\ m\equiv 1 \mod{2} }}\kappa_m\frac{e\left(g^{*}\left(\frac{2^{j}q_0s_1m}{d}\right)+\frac{b_3 m^3}{\ell}+\gamma'_0 m\right)}{(2^j m \ell)^{1/2}},
\end{align*}
for some integer $b_3$, real number $\gamma'_0\in \R$ and sequence $\kappa_m$ satisfying $\kappa_m\ll 1$ which is periodic mod $4$.  Combining the above with~\eqref{eq:S1dpart-2} and  taking a maximum over summation in $j$, there exists $b_4,\gamma_0''$ and $j\ll \log{N}$ such that
\begin{align*}
& S^{(1)}_{d}\ll N^{o(1)}\left(\frac{d}{\ell 2^{j}}\right)^{1/2} \\ & \quad \quad \quad \quad \times \sum_{\substack{dY/2^{j}q_0s_1\le m\le dN/2^{j}q_0s_1 \\ (m,sq_0)=1 \\ m\equiv 1 \mod{2} }}\kappa_m\frac{e\left(g^{*}\left(\frac{2^{j}q_0s_1m}{d}\right)+\frac{b_4 m^3}{\ell}+\gamma''_0 m\right)}{m^{1/2}}.
\end{align*}
Recalling~\eqref{eq:g*def} and partitioning summation over $m$ into residue classes mod $4\ell$ and taking a maximum, we see that there exists some integer $u$ such that 
\begin{align*}
& S^{(1)}_{d}\ll N^{o(1)}\left(\frac{d\ell }{s_1^22^{j}}\right)^{1/2} \\ & \quad \quad \quad \quad \times\sum_{\substack{dY/2^{j}q_0s_1\le m\le dN/2^{j}q_0s_1 \\ (m,sq_0)=1 \\ m\equiv u \mod{4\ell} }}\frac{1}{m^{1/2}}e\left(\frac{1}{4}g\left(\frac{2^{j}q_0s_1m}{d}\right)+\gamma'_0 m\right),
\end{align*}
for suitable $\gamma_0'$. The result follows from Lemma~\ref{lem:apcomp} after renaming $\gamma_0'$.
\end{proof}
We next simplify the sums $S^{(2)}_d$. Inspecting the the proof of Lemma~\ref{lem:weylS2}, it is sufficient to restrict the parameter $j$ to either $0,1$, although the statement below keeps presentation in line with Lemma~\ref{lem:weylS1}. 

\begin{lemma}
\label{lem:weylS2}
With notation and conditions as in Lemma~\ref{lem:weyl4}, there exists a positive integer $j$ and a real number $\gamma_1$ such that
\begin{align*}
& S^{(2)}_{d}\ll N^{o(1)}\left(1+\frac{|s|q}{\ell d^2 Y}\right)\left(\frac{d\ell}{2^j}\right)^{1/2} \\ & \quad \quad \quad \times\sum_{\substack{dY/2^{j}q_0s_1\le m\le dN/2^{j}q_0s_1 \\ (m,sq_0)=1  }}\frac{1}{m^{1/2}}e\left(\frac{1}{4}g\left(\frac{2^{j}q_0s_1m}{d}\right)+\gamma_1 m\right).
\end{align*}

\end{lemma}
\begin{proof}
Our proof is similar to Lemma~\ref{lem:weylS1} with some extra technical details. Recalling~\eqref{eq:Sd2}
\begin{align*}
& S^{(2)}_d= \\ & \sum_{\substack{dY/q_0s_1\le |m|\le dN/q_0s_1 \\ (m,q_0)=1 \\ (m,s)=1}}e\left(g^{**}\left(\frac{q_0s_1m}{d}\right)\right) \frac{G(\overline {bs_1m},\overline{bs_1m};d)G\left(s_2\overline d,s_2\overline{d}-\frac{s_1tq_0m^2}{d};\ell m\right)}{m\ell},
\end{align*}
and as before
\begin{align}
\label{eq:S2dpart-2}
S^{(2)}_d\ll\sum_{j\ll \log{N}}S^{(2)}_{d,j,+}+S^{(2)}_{d,j,-},
\end{align}
where
\begin{align}
\label{eq:S2djpart-21}
\nonumber  S^{(2)}_{d,j,\pm}
&=\sum_{\substack{dY/2^{j}q_0s_1\le m\le dN/2^{j}q_0s_1 \\ (2^{j}m,sq_0)=1 \\ m\equiv 1 \mod{2} }}e\left(g^{**}\left(\pm\frac{2^{j}q_0s_1m}{d}\right)\right)  \\ & \quad \quad \quad \times \frac{G(\pm\overline {bs_12^{j}m},\pm\overline{bs_12^{j}m};d)G\left(s_2\overline d,s_2\overline{d}-\frac{s_1tq_02^{j}m^2}{d};\pm\ell 2^{j}m\right)}{2^{j} m\ell}. 
\end{align}
We provide details only for $S^{(2)}_{d,j,+}.$ After taking complex conjugates, a similar argument applies to $S^{(2)}_{d,j,-}$. Conisder first the term 
\begin{align*}
G(\overline {bs_12^{j}m},\overline{bs_12^{j}m};d).
\end{align*}
For the right hand side of of summation in~\eqref{eq:S2djpart-21} to be nonzero, by Lemma~\ref{lem:evenaa1} we may suppose either 
\begin{align*}
d\equiv 1 \mod{4}, \quad d\equiv 3 \mod{4} \quad \text{or} \quad d\equiv 2 \mod{4}.
\end{align*}
If either $d\equiv 1 \mod{4}$ or $d\equiv 3 \mod{4}$ then by Lemma~\ref{lem:qreciprocity}, Lemma~\ref{lem:gausssum2} and Lemma~\ref{lem:13}
\begin{align}
\label{eq:Godd123}
\nonumber G(\overline {bs_12^{j}m},\overline{bs_12^{j}m};d)&=e_d(-\overline{4bs_1 2^{j}m})G(\overline {bs_12^{j}m},0;d) \\
&=\varepsilon_d e_d(-\overline{4bs_1 2^{j}m})\left(\frac{bs_1 2^{j}m}{d}\right)d^{1/2} \\
&=\varepsilon_d e_d(-\overline{4bs_1 2^{j}m})\left(\frac{bs_1 2^{j}}{d}\right)\left(\frac{d}{m}\right)(-1)^{(m-1)(d-1)/4}d^{1/2}, \nonumber 
\end{align}
with $\varepsilon_d$ as in~\eqref{eq:epsddef}. If $d\equiv 2 \mod{4}$ then by Lemma~\ref{lem:evenaa}
\begin{align*}
G(\overline {bs_12^{j}m},\overline{bs_12^{j}m};d)=2G\left(2\overline {bs_12^{j}m},\overline{bs_12^{j}m};\frac{d}{2}\right).
\end{align*}
Since $d/2$ is odd, by Lemma~\ref{lem:qreciprocity}, Lemma~\ref{lem:gausssum2} and Lemma~\ref{lem:13} 
\begin{align}
\label{eq:Geven123}
\nonumber G(\overline {bs_12^{j}m},\overline{bs_12^{j}m};d)&=2e_{d/2}(-\overline{8 bs_12^{j}m})G\left(2\overline {bs_12^{j}m},0;\frac{d}{2}\right)  \\
&=\varepsilon_{d/2}e_{d/2}(-\overline{8 bs_12^{j}m})\left(\frac{2bs_12^{j}m}{d/2} \right)(2d)^{1/2} \nonumber \\
&=\varepsilon_{d/2}e_{d/2}(-\overline{8 bs_12^{j}m})\left(\frac{2bs_12^{j}}{d/2} \right)\left(\frac{d/2}{m}\right) \\  & \quad \quad \quad \quad \times (-1)^{(m-1)(d/2-1)/4}(2d)^{1/2}. \nonumber 
\end{align}
Define 
\begin{align}
\label{eq:d0def}
d_0=\begin{cases} 
d \quad \text{if} \quad d\equiv 1 \mod{2} \\
\frac{d}{2} \quad \text{if} \quad d\equiv 2 \mod{4},
\end{cases}
\end{align}
and 
\begin{align}
\label{eq:etadef-99}
\eta= \begin{cases} 1 \quad \text{if} \quad d\equiv 1 \mod{2} \\
2 \quad \text{if} \quad d\equiv 2 \mod{4}.
\end{cases}
\end{align}
Combining~\eqref{eq:S2djpart-21},~\eqref{eq:Godd123} and~\eqref{eq:Geven123}, we see that there exists some $\gamma_2\in \R$ such that 
\begin{align}
\label{eq:S2djpart-2112}
\nonumber & S^{(2)}_{d,j,+}\ll d^{1/2}\sum_{\substack{dY/2^{j}q_0s_1\le m\le dN/2^{j}q_0s_1 \\ (2^{j}m,sq_0)=1 \\ m\equiv 1 \mod{2} }}e\left(g^{**}\left(\frac{2^{j}q_0s_1m}{d}\right)+\gamma_2 m\right)e_{d_0}(-\overline{4\eta b s_1 2^{j} m})  \\ & \quad \quad \quad \times \left(\frac{d_0}{m}\right)\frac{G\left(s_2\overline {\eta d_0},s_2\overline{\eta d_0}-\frac{s_1tq_02^{j}m^2}{\eta d_0};\ell 2^{j}m\right)}{2^{j} m\ell}.
\end{align}
We have 
\begin{align}
\label{eq:G-978675}
G\left(s_2\overline {\eta d_0},s_2\overline{\eta d_0}-\frac{s_1tq_02^{j}m^2}{\eta d_0};\ell 2^{j}m\right)=G\left( \overline{s_2}\eta d_0,1-s_1\overline{s_2}tq_02^{j}m^2;\ell  2^{j}m\right).
\end{align}
We next proceed on a case by case basis depending on whether $j=0$ or not. If $j\ge 1$ then by~\eqref{eq:G-978675} and Lemma~\ref{lem:gausssum1}, there exists some integer $b_4$ and real number $\gamma_2'$ such that
\begin{align*}
&G\left(s_2\overline {\eta d_0},s_2\overline{\eta d_0}-\frac{s_1tq_02^{j}m^2}{\eta d_0};\ell 2^{j}m\right)= \\ & \quad \quad \quad \quad \quad \quad \quad \quad \quad \quad e\left(\frac{b_4 m^3}{2\ell}+\gamma_2' m \right)G(s_2 \overline{\eta d_0}, s_2 \overline{\eta d_0};\ell 2^{j}m).
\end{align*}
By Lemma~\ref{lem:evenaa1}, if summation in~\eqref{eq:S2djpart-2112} is nonzero we must have $j=1$ and $\ell  m$ odd. Hence it is sufficient to estimate $S^{(2)}_{j,1,+}$ with the condition $\ell $ odd.
By the above, Lemma~\ref{lem:gausssum2} and Lemma~\ref{lem:evenaa} 
\begin{align*}
G\left(s_2\overline {\eta d_0},s_2\overline{\eta d_0}- \frac{2s_1tq_0m^2}{\eta d_0};2\ell  m\right)=2e\left(\frac{b_4 m^3}{\ell}+\gamma_2' m \right)G(2 s_2 \overline{\eta d_0}, s_2 \overline{\eta d_0};\ell m) \\ 
 =2e\left(\frac{b_4 m^3}{\ell}+\gamma_2' m-\frac{\overline{2^{3}\eta d_0}s_2}{\ell m} \right)G(2 s_2 \overline{\eta d_0},0;\ell m).
\end{align*}
Substituting into~\eqref{eq:S2djpart-2112} and recalling $j=1$,~\eqref{eq:ellsdef} and~\eqref{eq:s2def11}
\begin{align}
\label{eq:s1234-22}
\nonumber & S^{(2)}_{d,1,+}\ll d^{1/2}\sum_{\substack{dY/2q_0s_1\le m\le dN/2q_0s_1 \\ (2m,sq_0)=1 \\ m\equiv 1 \mod{2} }}e\left(g^{**}\left(\frac{2q_0s_1m}{d}\right)+\frac{b_4 m^3}{\ell}+\gamma_2 m\right)  \\ & \quad \quad \quad \times e\left(-s_2\left(\frac{\overline{2^3\eta \ell m}}{d_0}+\frac{\overline{2^{3}\eta d_0}}{\ell m}\right)\right)\left(\frac{d_0}{m}\right)\frac{G(2 s_2 \overline{\eta d_0},0;\ell m)}{ m\ell}.
\end{align}
Using reciprocity for modular inverses, we have 
\begin{align*}
\frac{\overline{2^3\eta \ell m}}{d_0}+\frac{\overline{2^{3}\eta d_0}}{\ell m}\equiv -\frac{\overline{d_0}}{2^{3}\eta \ell m}+\frac{\overline{2^{3}\eta d_0}}{\ell m}+\frac{1}{d_02^3\eta \ell m} \mod{1},
\end{align*}
and since 
\begin{align*}
\frac{\overline{d_0}}{2^{3}\eta \ell m}\equiv \frac{\overline{2^{3}\eta d_0}}{\ell m}+\frac{\overline{\ell m d_0}}{2^{3}\eta} \mod{1},
\end{align*}
we get 
\begin{align}
\label{eq:recmmin}
\frac{\overline{2^3\eta \ell m}}{d_0}+\frac{\overline{2^{3}\eta d_0}}{\ell m}\equiv -\frac{\overline{\ell m d_0}}{2^{3}\eta}+\frac{1}{d_02^3\eta \ell m} \mod{1}.
\end{align}
Using the above in~\eqref{eq:s1234-22}, recalling~\eqref{eq:etadef-99} and partitioning summation over $m$ into residue classes mod $16$, there exists some integer $\nu_1$ such that 
\begin{align}
\label{eq:s1234-2211}
\nonumber & \sum_{j\ge 1}S^{(2)}_{d,j,+}\ll d^{1/2}\sum_{\substack{dY/2q_0s_1\le m\le dN/2q_0s_1 \\ (2m,sq_0)=1 \\ m\equiv \nu_1 \mod{16} }}e\left(g^{**}\left(\frac{2q_0s_1m}{d}\right)+\frac{b_4 m^3}{\ell}+\gamma_2 m+\frac{s_2}{d_02^3\eta \ell m}\right)  \\ & \quad \quad \quad \quad \quad \quad \quad \quad \quad \quad  \times \left(\frac{d_0}{m}\right)\frac{G(2 s_2 \overline{\eta d_0},0;\ell m)}{ m\ell}.
\end{align}
Using Lemma~\ref{lem:13} and noting summation over $m$ in~\eqref{eq:s1234-2211} is restricted to a fixed congruence class mod $16$ we get  
\begin{align*}
\nonumber & \sum_{j\ge 1}S^{(2)}_{d,j,+}\ll \left(\frac{d}{\ell }\right)^{1/2} \\ & \quad \quad  \quad \times \sum_{\substack{dY/2q_0s_1\le m\le dN/2q_0s_1 \\ (2m,sq_0)=1 \\ m\equiv \nu_1 \mod{16} }}\frac{1}{m^{1/2}}e\left(g^{**}\left(\frac{2q_0s_1m}{d}\right)+\frac{b_4 m^3}{\ell}+\gamma_2 m+\frac{s_2}{d_02^3\eta \ell m}\right).
\end{align*}
Using Lemma~\ref{lem:partialsummation} to remove the term $s_2/(d_02^3\eta \ell m)$, the above simplifies to
\begin{align*}
\nonumber & \sum_{j\ge 1}S^{(2)}_{d,1,+}\ll N^{o(1)}\left(1+\frac{|s|q}{\ell d^2Y}\right)\left(\frac{d}{\ell }\right)^{1/2} \\ & \quad \quad  \quad \times \sum_{\substack{dY/2q_0s_1\le m\le dN/2q_0s_1 \\ (2m,sq_0)=1 \\ m\equiv \nu_1 \mod{16} }}\frac{1}{m^{1/2}}e\left(g^{**}\left(\frac{2q_0s_1m}{d}\right)+\frac{b_4 m^3}{\ell}+\gamma'_2 m\right),
\end{align*}
for some real number $\gamma_2'$. Finally, partitioning summation into residue classes mod $\ell$, recalling~\eqref{eq:g**def} and using Lemma~\ref{lem:apcomp} as in the proof of Lemma~\ref{lem:weylS1} we obtain 
\begin{align}
\label{eq:S2d111-final}
\nonumber & \sum_{j\ge 1}S^{(2)}_{d,1,+}\ll N^{o(1)}(d \ell )^{1/2}\left(1+\frac{|s|q}{\ell d^2Y} \right) \\ & \quad \quad  \quad \times \sum_{\substack{dY/2q_0s_1\le m\le dN/2q_0s_1 \\ (m,sq_0)=1}}\frac{1}{m^{1/2}}e\left(\frac{1}{4}g\left(\frac{2q_0s_1m}{d}\right)+\gamma_1 m\right),
\end{align}
for some real number $\gamma_1$.  

Returning to~\eqref{eq:S2djpart-2112}, consider next when $j=0$. Using~\eqref{eq:G-978675}, we have  
\begin{align}
\label{eq:abc-645}
\nonumber & S^{(2)}_{d,0,+}\ll d^{1/2}\sum_{\substack{dY/q_0s_1\le m\le dN/q_0s_1 \\ (m,sq_0)=1 \\ m\equiv 1 \mod{2} }}e\left(g^{**}\left(\frac{q_0s_1m}{d}\right)+\gamma_2 m\right)e_{d_0}(-\overline{4\eta b s_1  m})  \\ & \quad \quad \quad \times \left(\frac{d_0}{m}\right)\frac{G\left( \overline{s_2}\eta d_0,1-s_1\overline{s_2}tq_0m^2;\ell  m\right)}{m\ell}.
\end{align}
Let $2^{r}$ be the largest power of $2$ dividing $\ell$ and write 
\begin{align}
\label{eq:ell1ell23}
\ell=2^{r}\ell_1,
\end{align}
with $\ell_1$ odd. By Lemma~\ref{lem:gaussmult}
\begin{align*}
&G\left( \overline{s_2}\eta d_0,1-s_1\overline{s_2}tq_0m^2,\ell m\right)= \\ &  \quad \quad \quad G\left( \overline{s_2}\eta d_0 \ell_1 m,1-s_1\overline{s_2}tq_0m^2,2^{r}\right)G\left(2^{r} \overline{s_2}\eta d_0,1-s_1\overline{s_2}tq_0m^2,\ell_1 m\right).
\end{align*}
From the above and~\eqref{eq:abc-645}, there exists integers $\nu_1,\nu_2,\nu_3$ with $\nu_1,\nu_2$ odd such that 
\begin{align}
\label{eq:abc-6451}
\nonumber & S^{(2)}_{d,0,+}\ll d^{1/2}G(\nu_2,\nu_3,2^{r})\sum_{\substack{dY/q_0s_1\le m\le dN/q_0s_1 \\ (m,sq_0)=1  \\ m\equiv \nu_1 \mod{2^{r+3}} }}e\left(g^{**}\left(\frac{q_0s_1m}{d}\right)+\gamma_2 m\right)e_{d_0}(-\overline{4\eta b s_1  m})  \\ & \quad \quad \quad \times \left(\frac{d_0}{m}\right)\frac{G\left(2^{r} \overline{s_2}\eta d_0,1-s_1\overline{s_2}tq_0m^2,\ell_1 m\right)}{m\ell_1}.
\end{align}
With notation as in~\eqref{eq:epsddef}, by Lemma~\ref{lem:gausssum2} and Lemma~\ref{lem:13}
\begin{align*}
& G\left(2^{r} \overline{s_2}\eta d_0,1-s_1\overline{s_2}tq_0m^2,\ell_1 m\right) \\
& \quad \quad \quad \quad \quad =e_{\ell_1m}\left(s_2 \overline{\eta d_0 2^{r+2}}(1-s_1\overline{s_2}t q_0 m^2)^2 \right)G(2^{r}\overline{s_2}\eta d_0,0;\ell_1m) \\
& \quad \quad \quad \quad \quad =\varepsilon_{\ell_1m} e_{\ell_1m}\left(-s_2 \overline{\eta d_0 2^{r+2}}(1-s_1\overline{s_2}t q_0 m^2)^2 \right)\left( \frac{2^{r} s_2 \eta d_0}{\ell_1 m}\right)(\ell_1 m)^{1/2}.
\end{align*}
Substituting into~\eqref{eq:abc-6451}, there exists an integer $b_5$ and real number $\gamma_2'$ such that
\begin{align}
\label{eq:S20-almostlast}
\nonumber & S^{(2)}_{d,0,+}\ll \left(\frac{2^{r}d}{\ell_1}\right)^{1/2}\sum_{\substack{dY/q_0s_1\le m\le dN/q_0s_1 \\ (m,sq_0)=1  \\ m\equiv \nu_1 \mod{2^{r+3}} }}\frac{1}{m^{1/2}}e\left(g^{**}\left(\frac{q_0s_1m}{d}\right)+\frac{b_5m^3}{\ell_1}+\gamma'_2 m\right)  \\ & \quad \quad \quad \times e\left(-\left(\frac{\overline {4\eta b s_1 m}}{d_0}+\frac{s_2\overline{\eta d_0 2^{r+2}}}{\ell_1m} \right) \right),
\end{align}
where we have used the bound 
$$G(\nu_2,\nu_3,2^{r})\ll 2^{r/2}.$$
From~\eqref{eq:ellsdef},~\eqref{eq:s2def11} and~\eqref{eq:ell1ell23}
\begin{align*}
\frac{\overline {4\eta b s_1 m}}{d_0}+\frac{s_2\overline{\eta d_0 2^{r+2}}}{\ell_1m}\equiv s_2\left(\frac{\overline {2^{r+2}\eta \ell_1 m}}{d_0}+\frac{\overline{2^{r+2}\eta d_0}}{\ell_1m}\right) \mod{1},
\end{align*}
and by reciprocity for modular inverses
\begin{align*}
\frac{\overline {4\eta b s_1 m}}{d_0}+\frac{s_2\overline{\eta d_0 2^{r+2}}}{\ell_1m}&\equiv s_2\left(-\frac{\overline {d_0}}{2^{r+2}\eta \ell_1 m}+\frac{\overline{2^{r+2}\eta d_0}}{\ell_1m}+\frac{1}{d_02^{r+2}\eta \ell_1 m}\right) \mod{1} \\ 
&\equiv -\frac{s_2\overline{d_0 \ell_1 m}}{2^{r+2}\eta}+\frac{s_2}{4d_0\eta \ell m} \mod{1}.
\end{align*}
Substituting into~\eqref{eq:S20-almostlast} and using Lemma~\ref{lem:partialsummation} as before, we arrive at 
\begin{align}
\label{eq:S20-almostlast}
\nonumber & S^{(2)}_{d,0,+}\ll N^{o(1)}\left(1+\frac{|s|q}{\ell d^2 Y}\right)\left(\frac{2^{r}d}{\ell_1}\right)^{1/2} \\ & \quad \quad \times \sum_{\substack{dY/q_0s_1\le m\le dN/q_0s_1 \\ (m,sq_0)=1  \\ m\equiv \nu_1 \mod{2^{r+3}} }}\frac{1}{m^{1/2}}e\left(g^{**}\left(\frac{q_0s_1m}{d}\right)+\frac{b_5m^3}{\ell_1}+\gamma''_2 m\right),
\end{align}
for some real number $\gamma_2''$. Partitioning summation into residue classes mod $\ell_1$ and using Lemma~\ref{lem:apcomp} as in the proof of Lemma~\ref{lem:weylS1} gives
\begin{align*}
\nonumber & S^{(2)}_{d,0,+}\ll N^{o(1)}\left(1+\frac{|s|q}{\ell d^2 Y}\right)(d\ell )^{1/2}\sum_{\substack{dY/q_0s_1\le m\le dN/q_0s_1 \\ (m,sq_0)=1 }}\frac{1}{m^{1/2}}e\left(\frac{1}{4}g\left(\frac{q_0s_1m}{d}\right)+\gamma''_1 m\right),
\end{align*}
and the result follows combining the above with~\eqref{eq:S2dpart-2},~\eqref{eq:S2d111-final} and renaming terms. 
\end{proof}
Combining Lemma~\ref{lem:weyl4}, Lemma~\ref{lem:weylS1} and Lemma~\ref{lem:weylS2} and taking a maxmimum over $S^{(1)}_d$ and $S^{(2)}_d$ we deduce the following.
\begin{lemma}
\label{lem:weyl5}
With notation and conditions as in Lemma~\ref{lem:weyl4}, there exists a positive integer $j$ and some $\gamma_1\in \R$  such that 
\begin{align*}
& S_{d}\ll \left(\frac{dM_dN^{14\varepsilon}}{N}\right)^2\left(1 +\frac{q|s|}{\ell  N^2Y}\right)\left(1+\frac{q|s|}{\ell d^2 Y}\right)\frac{(d\ell )^{1/2}}{2^{j/2}} \\ & \quad \quad \quad \times \sum_{\substack{dY/2^{j}q_0s_1\le m\le dN/2^{j}q_0s_1 \\ (m,sq_0)=1  }}\frac{1}{m^{1/2}}e\left(\frac{1}{4}g\left(\frac{2^{j}q_0s_1m}{d}\right)+\gamma_1 m\right).
\end{align*}
\end{lemma}
\begin{lemma}
\label{lem:weyl6}
With notation and conditions as in Lemma~\ref{lem:weyl4}, there exists integers $a_1,d,M,y$ satisfying 
\begin{equation}
\label{eq:weyl6-conds}
\frac{Y}{2s}\le M_1 \le 2N, \quad d|q_0, \quad (a_1,d)=1,
\end{equation} 
and some $\gamma_1 \in \R$ such that writing 
\begin{align}
\label{eq:q1def}
q_1=(q,3)d,
\end{align}
we have 
\begin{align*}
S_{d}&\ll \left(\frac{dM_dN^{16\varepsilon}}{N}\right)^2\left(1 +\frac{q |s|}{\ell  N^2Y}\right)\left(1+\frac{q|s|}{\ell d^2 Y}\right)\left(\frac{q\ell}{M_1}\right)^{1/2} \\ & \quad \quad \quad \quad \quad \times  \sum_{\substack{1\le m\le q_1M_1/yq}}e\left(\frac{a_1 y^3q^2 m^3}{4q_1^3}+\gamma_1m\right).
\end{align*}
\end{lemma}
\begin{proof}
Let 
\begin{equation}
\label{eq:weyl6-Sdef5}
S=\sum_{\substack{dY/2^{j}q_0s_1\le m\le dN/2^{j}q_0s_1 \\ (m,sq_0)=1  }}\frac{1}{m^{1/2}}e\left(\frac{1}{4}g\left(\frac{2^{j}q_0s_1m}{d}\right)+\gamma_1 m\right),
\end{equation}
so by Lemma~\ref{lem:weyl5}
\begin{equation}
\label{eq:SS6-weyl5}
S_{d}\ll \left(\frac{dM_dN^{14\varepsilon}}{N}\right)^2\left(1 +\frac{q |s|}{\ell  N^2Y}\right)\left(1+\frac{q|s|}{\ell d^2 Y}\right)\frac{(d\ell )^{1/2}}{2^{j/2}}S.
\end{equation}
 In~\eqref{eq:weyl6-Sdef5} we partition summation over $m$ into dyadic intervals and take a maximum to obtain some $M$ satisying 
\begin{align}
\label{eq:Mbounds-84592}
Y/2\le M \le 2N,
\end{align}
such that 
\begin{align}
\label{eq:weyl6-Sdef}
S\ll N^{o(1)}\sum_{\substack{dY/2^{j}q_0s_1\le m\le dN/2^{j}q_0s_1 \\ dM/2^{j+1}q_0s_1< m\le dM/2^{j}q_0s_1 \\  (m,sq_0)=1   }}\frac{1}{m^{1/2}}e\left(\frac{1}{4}g\left(\frac{2^{j}q_0s_1m}{d}\right)+\gamma_1 m\right).
\end{align}
 By Lemma~\ref{lem:partialsummation} and Lemma~\ref{lem:intervalcompletion}
\begin{equation}
\label{eq:weyl6-SS'}
S\ll \frac{2^{j/2}q^{1/2}s_1^{1/2}N^{o(1)}}{d^{1/2}M^{1/2}}S',
\end{equation}
where 
\begin{align*}
S'=\sum_{\substack{1\le m\le dM/2^{j}q_0s_1 \\ (m,sq_0)=1  }}e\left(\frac{1}{4}g\left(\frac{2^{j}q_0s_1m}{d}\right)+\gamma'_1 m\right),
\end{align*}
for some $\gamma_1'\in \R$. Detecting the condition $(m,q_0s)$ via the M\"{o}bius function
\begin{align*}
S'&=\sum_{\substack{1\le m\le dM/2^{j}q_0s_1  }}\sum_{f|(m,q_0s)}\mu(f)e\left(\frac{1}{4}g\left(\frac{2^{j}q_0s_1m}{d}\right)+\gamma'_1 m\right) \\
&=\sum_{f|q_0s}\mu(f)\sum_{\substack{1\le m\le dM/2^{j}q_0s_1 \\ m\equiv 0 \mod{f}  }}e\left(\frac{1}{4}g\left(\frac{2^{j}q_0s_1m}{d}\right)+\gamma'_1 m\right).
\end{align*}
Taking a maximum over $f$, there exists some $f|q_0s$ such that 
\begin{align*}
S'&\ll q^{o(1)}\sum_{\substack{1\le m\le dM/2^{j}q_0s_1f  }}e\left(\frac{1}{4}g\left(\frac{2^{j}fq_0s_1m}{d}\right)+\gamma''_1 m\right),
\end{align*}
for some $\gamma''_1\in \R$. Recalling~\eqref{eq:gpoly} and~\eqref{eq:q0delta1}, we have 
\begin{align*}
S'\ll q^{o(1)}\sum_{\substack{1\le m\le dM/2^{j}q_0s_1f  }}e\left(\frac{a2^{3j}f^3q^2_0s^3_1m^3}{4(q,3)d^3}+\gamma'''_1m\right),
\end{align*}
for some $\gamma_1'''\in \R$. Let
\begin{equation}
\label{eq:weyl6-notation}
a_1=as_1^3, \quad y=2^{j}f, \quad q_1=(q,3)d, \quad M_1=\frac{M}{s_1}.
\end{equation}
Since $s$ satisfies~\eqref{eq:ellsdef} and $s_1|s$, we have 
\begin{align}
\label{eq:as1-495}
(as_1^3,q_1)=1,
\end{align}
 and from~\eqref{eq:Mbounds-84592}
\begin{align}
\label{eq:M1bounds}
\frac{Y}{2s}\le \frac{Y}{2s_1}\le M_1\le 2N.
\end{align}
By~\eqref{eq:as1-495} and~\eqref{eq:M1bounds}, the conditions~\eqref{eq:weyl6-conds} are satisfied. Using the notation~\eqref{eq:weyl6-notation} and recalling~\eqref{eq:q0delta1}, this gives 
\begin{align*}
S'\ll q^{o(1)}\sum_{\substack{1\le m\le q_1M_1/yq}}e\left(\frac{a_1q^2y^3m^3}{4q_1^3}+\gamma'''_1m\right).
\end{align*}
Combining the above with~\eqref{eq:SS6-weyl5} and~\eqref{eq:weyl6-SS'} results in
\begin{align*}
S_{d}&\ll \left(\frac{dM_dN^{16\varepsilon}}{N}\right)^2\left(1+\frac{q |s|}{\ell  N^2Y}\right)\left(1+\frac{q|s|}{\ell d^2 Y}\right)\left(\frac{q\ell }{M_1}\right)^{1/2} \\ & \quad \quad \quad \quad \times  \sum_{\substack{1\le m\le q_1M_1/yq}}e_{q_1}\left(\frac{a_1 y^3q^2 m^3}{q_1^2}+\gamma''''_1m\right),
\end{align*}
from which the desired result follows after renaming  $\gamma''''_1$.
\end{proof}
\section{Reduction to an iterative inequality}
We next combine Lemma~\ref{lem:weyliter} and~Lemma~\ref{lem:weyl6} to reduce to an iterative type inequality. 
\begin{lemma}
\label{lem:iterate}
For integers $a,q$ satisfying $(a,q)=1$ and $\gamma\in \R$ let 
\begin{align*} 
g(x)=\frac{a x^3}{q}+\gamma x.
\end{align*}
Let $\varepsilon,\rho,\delta>0$ be small and suppose $N$ satisfies
\begin{align}
\label{eq:deltaiter}
\frac{N^{1/2+\varepsilon}}{q^{1/4}}\ll \delta \le 1,
\end{align} 
and
\begin{align}
\label{eq:iterNconds}
q^{1/3}\le N \le \left(\delta^4q\right)^{1/(2+4\rho+22\varepsilon)}.
\end{align}
Let $f$ be a smooth function satisfying 
\begin{align*}
\text{supp}(f)\subseteq[N,2N], \quad f^{(j)}(x)\ll \frac{1}{x^{j}}.
\end{align*}
and suppose that 
\begin{align}
\label{eq:itercond}
\sum_{n\in \Z}f(n)e\left(g(n)\right) \gg \delta (qN)^{1/4+o(1)}.
\end{align}
There exists integers $a_1,q_1,M,y$ satisfying 
\begin{equation}
\label{eq:M-76-conds}
\frac{q}{q_1}\le M \le 2N, \quad q_1|q, \quad 1\le y\le \frac{q_1 M}{q}, \quad (a_1,q_1)=1,
\end{equation} 
\begin{align}
\label{eq:dcase2}
q_1\gg \delta^2q,
\end{align}
and some $\gamma_1 \in \R$ such that
\begin{align*}
\delta^4 (NM)^{1/2+o(1)}&\ll N^{32\varepsilon}\left(1+\frac{q_1^4N^{4\rho}}{q^4\delta^6} \right)\left(1+\frac{N^{2+2\rho}}{\delta^4q}\right) \\ & \quad \quad \quad \quad \quad \quad \quad \quad \times  \sum_{\substack{1\le m\le q_1M/yq}}e\left(\frac{a_1 y^3q^2 m^3}{4q_1^3}+\gamma_1m\right).
\end{align*}
\end{lemma}
\begin{proof}
Let $d$ be as in Lemma~\ref{lem:weyliter} and recall the notation~\eqref{eq:bdef}. We first show the assumptions~\eqref{eq:deltaiter} and~\eqref{eq:itercond} imply~\eqref{eq:dcase2}. Recalling~\eqref{eq:q1def} it is sufficient to show 
$$d\gg \delta^2q.$$
Apply Lemma~\ref{lem:weyliter} to the sum
\begin{align*}
\sum_{n\in \Z}f(n)e\left(g(n)\right),
\end{align*}
 with $Y=N$. From~\eqref{eq:itercond}, this implies either  
\begin{equation}
\label{eq:iterprelim1}
\left|\sum_{n\in \Z}f(n)e\left(g(n)\right)\right|^2 \ll  N^{1+o(1)},
\end{equation}
or there exists integers $\ell_0,s_0$ satisfying
\begin{align}
\label{eq:iterprelim2}
b\equiv \ell_0 \overline{s_0} \mod{d}, \quad \ell_0\ll \frac{1}{\delta^4}\left(\frac{d}{q}\right)^2, \quad |s_0|\ll \frac{1}{\delta^4}\left(\frac{d}{q}\right)^2\frac{N^3}{q}.
\end{align}
By~\eqref{eq:deltaiter} and~\eqref{eq:itercond} the inequality~\eqref{eq:iterprelim1} does not hold, hence we must have~\eqref{eq:iterprelim2}. Since $(b,d)=1$, we have  $\ell_0 \ge1$ which implies~\eqref{eq:dcase2}.

With notation as in Lemma~\ref{lem:gausssum1}, for integer $j$ define 
\begin{align*}
S_{d,j}=\sum_{\substack{dN^{1-j\rho}/q\le |m|\le dN/q \\ (m,q_0)=1}}e\left(g\left(\frac{q_0m}{d}\right)\right)\sum_{n\in \Z}F_{d,m}(n)e(\rho_{q_0m/d}n)e_{d}(bmn^2),
\end{align*}
and let $j_0$ denote the smallest nonnegative integer such that 
\begin{align}
\label{eq:ub123-end}
\left|\sum_{n\in \Z}f(n)e\left(g(n)\right)\right|^2 \ll q^{o(1)}S_{d,j_0}.
\end{align}
 By Lemma~\ref{lem:weyl3}, $j_0$ exists and satisfies $j_0\le \rho^{-1}$ and $j_0\ge 1$ since otherwise 
\begin{align*}
\left|\sum_{n\in \Z}f(n)e\left(g(n)\right)\right|^2\ll N^{1+o(1)},
\end{align*}
contradicting~\eqref{eq:deltaiter} and~\eqref{eq:itercond}. Apply Lemma~\ref{lem:weyliter} to the sum
\begin{align*}
\sum_{n\in \Z}f(n)e\left(g(n)\right),
\end{align*}
 with 
\begin{align}
\label{eq:YYYYY}
Y=N^{1-(j_0-1)\rho}.
\end{align}
Either there exists $\ell,s$ satisfying
\begin{align}
\label{eq:itercase1-1}
b\equiv \ell s^{-1} \mod{d}, \quad \ell \ll \frac{1}{\delta^4}\left(\frac{d}{q}\right)^2N^{-(j_0-1)\rho}, \quad |s|\ll \frac{1}{\delta^4}\left(\frac{d}{q}\right)^2\frac{N^{3-2(j_0-1)\rho}}{q},
\end{align}
or 
\begin{align}
\label{eq:itercase2-1}
\left|\sum_{n\in \Z}f(n)e\left(g(n)\right)\right|^2 \ll q^{o(1)}S_{d,j_0-1}.
\end{align}
By definition of $j_0$,~\eqref{eq:itercase2-1} does not hold and hence we must have~\eqref{eq:itercase1-1}. Writing 
\begin{align}
\label{eq:Sdj0def66}
S_d=S_{d,j_0}=\sum_{\substack{dN^{1-j\rho_0}/q\le |m|\le dN/q \\ (m,q_0)=1}}e\left(g\left(\frac{q_0m}{d}\right)\right)\sum_{n\in \Z}F_{d,m}(n)e(\rho_{q_0m/d}n)e_{d}(bmn^2),
\end{align}
we apply Lemma~\ref{lem:weyl6}  with 
$$Y=N^{1-j_0 \rho},$$
and $\ell,s$ as in~\eqref{eq:itercase1-1}. Recalling~\eqref{eq:Mjdef}, these parameters give
\begin{align*}
M_d=\max\left\{\frac{1}{\ell ^{1/2}\delta^2}\frac{N^{1+\rho}}{q},\frac{N}{d}\right\}.
\end{align*} 
This implies that 
\begin{align}
\label{eq:Md11}
\frac{dM_d}{N}\ll \left(1+\frac{dN^{\rho}}{q\ell ^{1/2}\delta^2} \right).
\end{align}
We also have 
\begin{align}
\label{eq:otherterm123}
\left(1+\frac{q |s|}{\ell  N^2Y}\right)\ll \left(1 +\frac{d^2N^{2\rho}}{\delta^4\ell q^2}\right),
\end{align}
and 
\begin{align}
\label{eq:otherterm1234}
\left(1+\frac{q|s|}{\ell d^2 Y}\right)\ll \left(1+\frac{N^{2+2\rho}}{\delta^4\ell q}\right).
\end{align}
From~\eqref{eq:iterNconds},~\eqref{eq:itercase1-1} and~\eqref{eq:Md11}
\begin{align}
\label{eq:theotherterm}
\ell dM_d^2N^{11\varepsilon}\ll \frac{N^{2+2\rho+11\varepsilon}}{d}\left(\ell +\frac{1}{\delta^4} \right)\ll \frac{N^{2+2\rho+11\varepsilon}}{q\delta^4}<1,
\end{align}
hence the condition~\eqref{eq:Mjdconds} is satisfied. By Lemma~\ref{lem:weyl6},~\eqref{eq:Md11},~\eqref{eq:otherterm123} and~\eqref{eq:otherterm1234},
there exists integers $a_1,q_1,M,y$ satisfying 
\begin{equation}
\label{eq:weyl6-conds}
\frac{N^{1-j\rho}}{2|s|}\le M \le 2N, \quad q_1|q, \quad 1\le y\le \frac{q_1 M}{q_0}, \quad (a_1,q_1)=1,
\end{equation} 
and some $\gamma_1 \in \R$ such that
\begin{align*}
S_{d}&\ll N^{32\varepsilon}\left(1+\frac{q_1^4N^{4\rho}}{q^4\ell^2 \delta^8} \right)\left(1+\frac{N^{2+2\rho}}{\delta^4\ell q}\right)\left(\frac{q\ell}{M_1}\right)^{1/2} \\ & \quad \quad \quad \quad \quad \times  \sum_{\substack{1\le m\le q_1M_1/yq}}e\left(\frac{a_1 y^3q^2 m^3}{4q_1^3}+\gamma_1m\right).
\end{align*}
Combining the above with~\eqref{eq:itercond},~\eqref{eq:itercase1-1}~\eqref{eq:ub123-end} and~\eqref{eq:Sdj0def66} we obtain
\begin{align*}
\delta^4 (NM)^{1/2+o(1)}&\ll N^{32\varepsilon}\left(1+\frac{q_1^4N^{4\rho}}{q^4\delta^6} \right)\left(1+\frac{N^{2+2\rho}}{\delta^4q}\right) \\ & \quad \quad \quad \quad \quad \quad \quad \quad \times  \sum_{\substack{1\le m\le q_1M/yq}}e\left(\frac{a_1 y^3q^2 m^3}{4q_1^3}+\gamma_1m\right),
\end{align*}
which completes the proof.
\end{proof}
\begin{cor}
\label{cor:main}
For integers $a,q$ satisfying with $(a,q)=1$ and $\gamma\in \R$ let
\begin{align*} 
g(x)=\frac{a x^3}{q}+\gamma x.
\end{align*}
Let $\varepsilon>0$ be small and suppose $N$ is an integer and $\delta$ a real number  satisfying 
\begin{align}
\label{eq:iterNconds1}
q^{1/3}\le N \le \delta^2q^{1/2-o(1)},
\end{align}
\begin{align}
\label{eq:deltacondsmain}
\frac{N^{1/2+o(1)}}{q^{1/4}}\le \delta < 1.
\end{align} 
Let $f$ be a smooth function satisfying 
\begin{align*}
\text{supp}(f)\subseteq[N,2N], \quad f^{(j)}(x)\ll \frac{1}{x^{j}},
\end{align*}
and suppose that
\begin{align}
\label{eq:itercondmain}
\sum_{n\in \Z}f(n)e\left(g(n)\right) \gg \delta (qN)^{1/4+o(1)}.
\end{align}
There exists integers $a_1,M,h,q_1$  satisfying 
\begin{align}
\label{eq:Mh1h2conds}
1\le M \le N, \quad h\le M, \quad q_1|4q, \quad (a_1,q_1)=1,
\end{align}
and
\begin{align}
\label{eq:q2condsmain}
q_1\gg \frac{\delta^2q}{h^3},
\end{align}
such that for some $\gamma_1\in \R$ and polynomial $g_1$ of the form
$$g_1(x)=\frac{a_1 x^3}{q_1}+\gamma_1x,$$
we have 
\begin{align*}
\delta^{10} (NM)^{1/2}\ll N^{o(1)}\sum_{\substack{1\le m\le M/h}}e\left(g_1(m)\right).
\end{align*}
\end{cor}
\begin{proof}
Let $\varepsilon_1$ be small and $\varepsilon,\rho,M,q_1,y,a_1$ be as in Lemma~\ref{lem:iterate}. Taking $\varepsilon,\rho$  sufficiently small and using~\eqref{eq:deltacondsmain}, we obtain from Lemma~\ref{lem:iterate}
\begin{align*}
\label{eq:deltaD1-99}
\delta^4 (NM)^{1/2}&\ll N^{40\varepsilon}\left(1+\frac{q_1^4}{q^4\delta^6} \right)\sum_{\substack{1\le m\le q_1M/yq}}e\left(\frac{a_1 y^3q^2 m^3}{4q_1^3}+\gamma_1m\right).
\end{align*}
provided $N$ satisfies~\eqref{eq:iterNconds1}. Define $h_1,h_2$ by 
\begin{align*}
h_1=y, \quad h_2=\frac{q}{q_1},
\end{align*}
and let 
$$d=(h_1^3h_2^2,4q_1), \quad q_2=\frac{4q_1}{d}.$$
With this notation,~\eqref{eq:deltaD1-99} becomes 
\begin{align*}
\delta^4 (NM)^{1/2}&\le N^{40\varepsilon}\left(1+\frac{1}{h_2^4\delta^6}\right)\sum_{\substack{1\le m\le M/h_1h_2}}e\left(\frac{a_2m^3}{q_2}+\gamma_1m\right) \\
& \le \frac{N^{40\varepsilon}}{\delta^6}\sum_{\substack{1\le m\le M/h_1h_2}}e\left(\frac{a_2m^3}{q_2}+\gamma_1m\right),
\end{align*}
for some $(a_2,q_2)=1$. Defining
$$h=h_1h_2,$$
and using~\eqref{eq:dcase2}, we have  
\begin{align*}
q_2\gg \frac{\delta^2q}{h^3},
\end{align*}
which establishes~\eqref{eq:q2condsmain} and the result follows after taking $\varepsilon$ sufficiently small and renaming $a_2,q_2$.
\end{proof}
\section{Proof of Theorem~\ref{thm:main1}}
 Suppose $N$ satisfies 
\begin{align}
q^{1/3+\rho}\le N \le q^{1/2-\rho/10+o(1)},
\end{align}
and define
\begin{align}
\label{eq:maindelta}
\delta = q^{-\rho/20}.
\end{align}
Assume for contradiction that 
\begin{align*}
\sum_{1\le n \le N}e\left(\frac{an^3}{q}+\gamma n \right) \gg \delta (qN)^{1/4+o(1)}.
\end{align*}
By Lemma~\ref{lem:intervalsmooth}, there exists some $K\le 2N,$ smooth function $f$ satisfying 
\begin{align*}
\text{supp}(f)\subseteq [K,2K], \quad f^{(j)}(x)\ll \frac{1}{x^{j}}, 
\end{align*}
and real number $\gamma'$ such that
\begin{align}
\label{eq:main11-0}
\frac{\delta N^{1/4+o(1)}}{K^{1/4}}(qK)^{1/4}\ll \sum_{n\in \Z}f(n)e(g(n)+\gamma'n)).
\end{align}
Define 
\begin{align}
\label{eq:d0-----}
\delta_0=\frac{\delta N^{1/4}}{K^{1/4}},
\end{align}
and note that 
\begin{align*}
\delta_0\gg \delta.
\end{align*}
Let $\varepsilon>0$ be small. We may suppose 
\begin{align*}
K\ge q^{1/3+\varepsilon},
\end{align*}
since otherwise 
\begin{align*}
\delta \le  q^{\varepsilon}\left(\frac{q}{N^{3}}\right)^{1/12}\le q^{-\rho/4+\varepsilon}.
\end{align*}
By Corollary~\ref{cor:main}, there exists integers $a_1,M,h,q_1$  satisfying 
\begin{align}
\label{eq:Mh1h2conds-0}
1\le M \le 2K, \quad h\le M, \quad q_1|4q, \quad (a_1,q_1)=1,
\end{align}
and
\begin{align}
\label{eq:q2condsmain-0}
q_1 \gg \frac{\delta^2 q}{h^3},
\end{align}
such that for some $\gamma'_1\in \R$ 
we have 
\begin{align}
\label{eq:s1}
\delta_0^{10} (KM)^{1/2}\ll N^{o(1)}\sum_{\substack{1\le m\le M/h}}e\left(\frac{a_1 m^3}{q_1}+\gamma'_1m\right).
\end{align}
If 
\begin{align*}
h\ge q^{\rho/2},
\end{align*}
then using the trivial bound in~\eqref{eq:s1}
\begin{align*}
\delta^{10}\ll \delta_0^{10}\ll \frac{N^{o(1)}}{h}\ll q^{2\varepsilon-\rho/2}
\end{align*}
Hence we may suppose $h\le Z$. By~\eqref{eq:q2condsmain-0}, this implies 
\begin{align}
\label{eq:q1UB-00n}
q_1\gg \delta^2q^{1-3\rho/2}.
\end{align}
If
\begin{align*}
\frac{M}{h}\le q_1^{1/3+\varepsilon},
\end{align*}
then using the trivial bound in~\eqref{eq:s1} 
\begin{align}
\label{eq:s1}
\delta^{10}&\ll N^{o(1)}\left(\frac{M}{h}\right)^{1/2}\frac{1}{N^{1/2}} \ll q^{2\varepsilon}\left(\frac{q}{N^3}\right)^{1/6} \ll q^{2\varepsilon-\rho/2}.
\end{align}
If
\begin{align}
\label{eq:Mhlb}
q_1^{1/3}<\frac{M}{h}\le q_1,
\end{align} 
then by~\eqref{eq:weyl}

\begin{align*}
\delta_0^{10} (KM)^{1/2}\ll q_1^{1/4}\left(\frac{M}{h}\right)^{1/4+o(1)}+\left(\frac{M}{h}\right)^{3/4+o(1)}.
\end{align*}
Combining the above with~\eqref{eq:d0-----},~\eqref{eq:q1UB-00n} and~\eqref{eq:Mhlb}
\begin{align*}
\delta^{10}\ll q^{2\varepsilon-\rho/2}+q^{2\varepsilon-1/12-\rho/4}\ll q^{2\varepsilon-\rho/2},
\end{align*}
since we may suppose 
\begin{align}
\label{eq:rhoass123456789}
\rho\le 1/6.
\end{align}
Finally consider when 
\begin{align*}
\frac{M}{h}\ge q_1.
\end{align*}
By~\cite[Corollary~1.2]{BPPSV} and~\eqref{eq:q1UB-00n}
\begin{align}
\label{eq:s1}
\delta^{10}\ll \frac{q^{\varepsilon}}{q_1^{1/3}}+q^{2\varepsilon-1/12-\rho/4}\ll \frac{q^{\varepsilon-1/3+\rho/2}}{\delta^{2/3}}+q^{2\varepsilon-\rho/2},
\end{align}
which implies that either 
\begin{align*}
\delta^{10}\ll q^{2\varepsilon-\rho/2},
\end{align*}
or 
\begin{align*}
\delta^{32/3}\le q^{\varepsilon-1/3+\rho/2}.
\end{align*}
Combining the above estimates we obtain the desired result after recalling~\eqref{eq:rhoass123456789}, taking $\varepsilon$ sufficiently small and renaming $\rho$.

\section*{Acknowledgement} 
During this work the author was  supported  by  Academy of Finland Grant~319180 and would like to thank the University of Turku for its hospitality. 

The author is currently supported by the Max-Planck  Institute  for  Mathematics.

The author would like to thank Kaisa Matom\"{a}ki, Igor Shparlinski and Tim Trudgian for useful comments and discussions.


\begin{thebibliography}{99}
\bibitem{BI}
E. Bombieri and H. Iwaniec, {\it On the order of $\zeta(1/2+it)$}, Ann. Scuola. Norm. Sup. di. Pisa, {\bf 143}, (1986), no. 3, 449--472.
\bibitem{Bou}
J. Bourgain, {\it Fourier transform restriction phenomena for certain lattice subsets and applications to nonlinear evolution equations. Part II: The KDV-Equation}, Geom. Funct. Anal. {\bf 3 } (1993), no. 3, 209--262.
\bibitem{BPPSV}
J. Brandes, S. T. Parsell, C. Poulias, G. Shakan, and R. C. Vaughan, {\it On generating functions in additive number theory, II: Lower-order terms and applications to PDEs}, arXiv:2001.05629.
\bibitem{BEW}
B. C. Berndt, R. J. Evans and K. S. Williams, {\it Gauss and Jacobi Sums}, John Wiley, New York, 1998.
\bibitem{CS}
C. Chen and I. E. Shparlinski, {\it On large values of Weyl sums}, Adv. Math., 2020, v.370, Art.107216.
\bibitem{CS1}
C. Chen and I. E. Shparlinski, {\it New bounds of Weyl sums}, . Int. Math. Res. Notices, (to appear). 
\bibitem{DKSZ}
 A. Dunn, B. Kerr, I. E. Shparlinski and A. Zaharescu, {\it Bilinear forms in Weyl sums for modular square roots and applications}, Adv. Math. (to appear).
\bibitem{Enflo}
P. Enflo, {\it Some problems in the interface between number theory, harmonic analysis
and geometry of Euclidian space}, First International Conference in Abstract Algebra,
Quaestiones Math. {\bf 18}, (1995),  no. 1-3, 309--323.
\bibitem{Fou1}
\'{E}. Fouvry, {\it Sur le probl\'{e}me des diviseurs de Titchmarsh,} J. reine angew. Math. {\bf 357} (1985), 51--76.
\bibitem{HL}
G. H. Hardy and J. E. Littlewood, `Some problems of diophantine approximation', Collected papers of G. H. Hardy, Oxford, (1966), 193--238.

\bibitem{HB}
 D. R. Heath-Brown,  {\it Bounds for the cubic Weyl sums} , J. Math. Sciences, {\bf171}, (2010), 813--823.
\bibitem{Hu}
Y. Hu and X. Li, {\it Discrete Fourier restriction associated with KdV equations}, Anal. PDE {\bf 6} (2013), no. 4, 859--892.
\bibitem{HuW}
K. Hughes and T. D. Wooley, {\it Discrete restriction for $(x,x^3)$ and related topics}, arXiv:1911.12262.

%\bibitem{IwKo} H. Iwaniec and E. Kowalski, {\it Analytic Number Theory}, Colloquium Publications {\bf 53} (2004), American Math. Soc., Providence, RI.
\bibitem{Lai}
X. Lai and Y. Ding, {\it A note on the discrete Fourier restriction problem}, Proc. Amer. Math. Soc. {\bf 146} (2018), no. 9, 3839--3846.
\bibitem{Nakai}
Y.  Nakai.  {\it A  penultimate  step  toward  cubic  theta-Weyl  sums}, Dev. Math., {\bf 8},  Number  theoretic  methods, Kluwer Acad. Publ., Dordrecht, (2002), 311--338.

\bibitem{Vau}
R. C. Vaughan, {\it The Hardy-Littlewood method}, Cambridge University Press, 1981.
\bibitem{Weyl}
 H. Weyl, {\it Uber die Gleichverteilung von Zahlen mod Eins}, Math. Ann., {\bf 77}, (1916), 313--352.
\bibitem{Wooley1}
 T. D. Wooley, {\it Mean value estimates for odd cubic Weyl sums}, Bull. London Math. Soc. {\bf 47} (2015), no. 6, 946--957.
\bibitem{Wooley}
T. D. Wooley, {\it The cubic case of the main conjecture in Vinogradov’s mean value theorem}, Adv. Math. {\bf 294} (2016), 532--561.

\end{thebibliography}
\end{document}